\documentclass{amsart}[12pt]
\usepackage[pdftex]{graphicx}
\usepackage{color}
\usepackage[colorlinks=true,%
    linkcolor=blue,%
    citecolor=blue,%
    filecolor=blue,%
    menucolor=blue,%
    urlcolor=blue]{hyperref}
\hypersetup{pdftitle={Entropy in Dimension One}}
\hypersetup{pdfauthor={William P. Thurston}}
\graphicspath{{draws/}{}}

\newcommand{\R}{{\mathbb R}}
\newcommand{\E}{{\mathbb E}}
\newcommand{\Z}{{\mathbb Z}}
\newcommand{\C}{{\mathbb C}}

\newcommand{\Q}{{\mathbb Q}}
\newcommand{\proj}{{\mathbb P}}
\newcommand{\abs}[1]{{\left| #1 \right|}}

\DeclareMathOperator{\PL}{PL}
\DeclareMathOperator{\GL}{GL}

\DeclareMathOperator{\Split}{Split}

\DeclareMathOperator{\Var}{Var}

\theoremstyle{plain}
\newtheorem{theorem}{Theorem}[section]

\newtheorem{lemma}[theorem]{Lemma}

\newtheorem{proposition}[theorem]{Proposition}

\newtheorem{question}[theorem]{Question}

\theoremstyle{definition}
\newtheorem{defn}[theorem]{Definition}

\theoremstyle{remark}
\newtheorem{remark}[theorem]{Remark}

\begin{document}
\title{Entropy in Dimension One}

\author{William P. Thurston}
\date{January 28, 2014}
\maketitle

\section{Introduction}

\begin{figure}[htbp]
\centering
\includegraphics[width=4.5in]{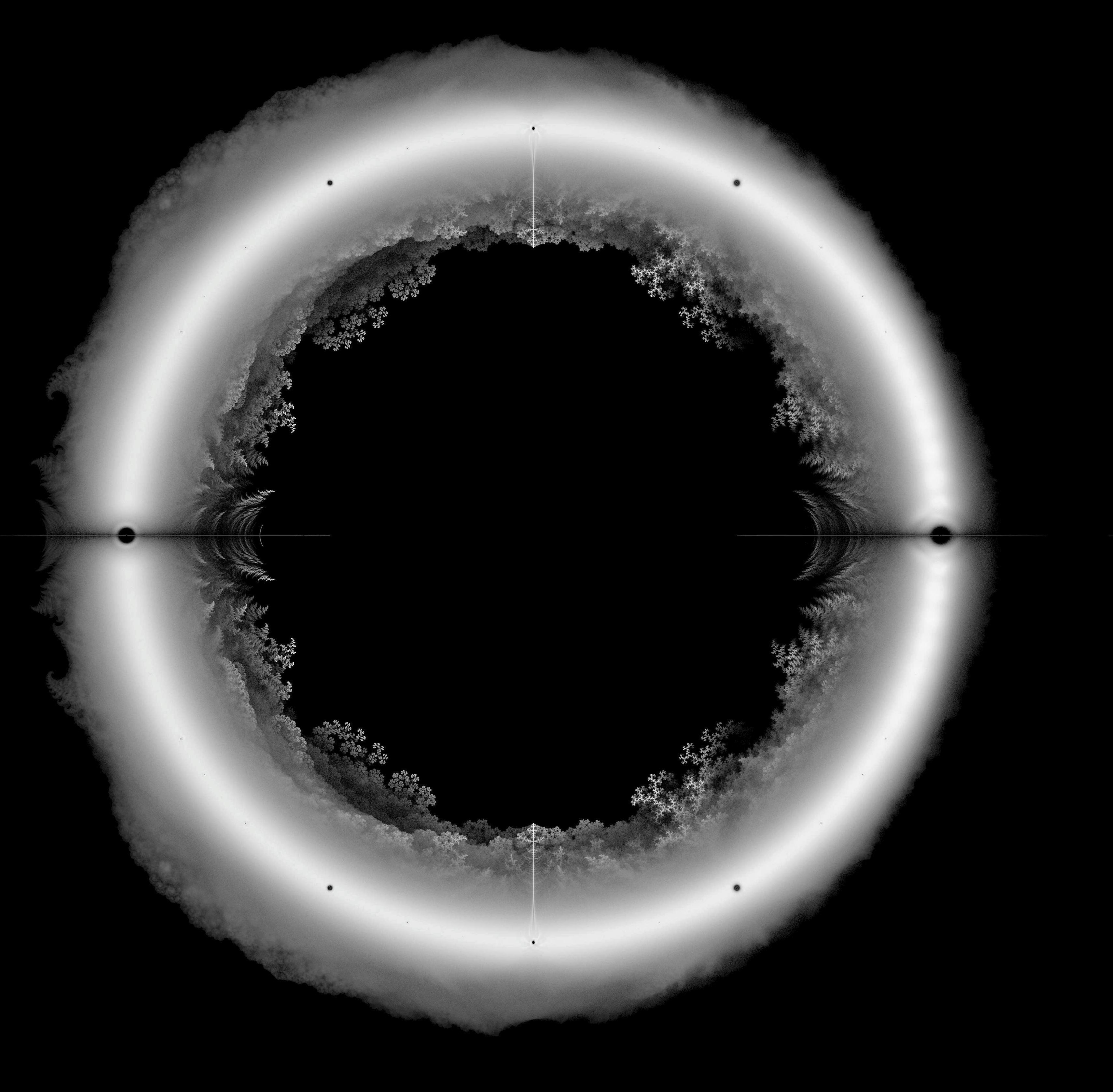}
\caption{ This is a plot of roughly $8*10^8$  roots of defining polynomials
for $\exp(h(f))$, where $f$ ranges over
a sample of about $10^7$ postcritically finite
quadratic maps of the interval with postcritical orbit of length $\le 80$.  The brightness is proportional to the log of the density; the highest concentration is at the unit circle. 
%The interval $[1,2]$ is contained in the set, but not very visible because of artifacts of binning.
}
\label{fig:KneadingRoots}
\end{figure}

The topological entropy $h(f)$ of a map from a compact topological space to itself, $f: X \rightarrow X$, is a numerical measure of the unpredictability of
trajectories $x, f(x), f^2(x), \dots $ of points under $f$: it is the
limiting upper bound for exponential growth rate of $\epsilon$-distinguishable
orbits, as $\epsilon \rightarrow 0$. Here $\epsilon$ can be measured with respect to an arbitrary metric on $X$, or one
can merely think of it as a neighborhood of the diagonal in $X \times X$, since all that matters is whether or not two points are within $\epsilon$ of each other.
The number of $\epsilon$-distinguishable orbits of length $n$ is the maximum cardinality of a set of orbits of $f$ such that no two are always within $\epsilon$.

 More
formally, given a metric $d$ on $X$, and a continuous map $f:X\to X$,  define the $\epsilon$-count  of  $X$, $N(X,d,\epsilon)$ to be the maximum cardinality of a set $S \subset X$ such that no point is within $\epsilon$ of any other. Define a metric $d_{f,n}$ on $X$
as $d_{f,n}(x, y) = \sup_{0 \le i < n} \left \{ d(f^i(x), f^i(y)) \right \}$.  Then

% Given a metric $d$ on $X$ and
%a continuous function $f:X \rightarrow X$, define a metric $d_{f,n}$ on $X$
%as $d_{f,n}(x, y) = \sup_{0 \le i < n} \left \{ d(f^i(x), f^i(y)) \right \}$.  Then

%$$
%h(f) = \lim_{\epsilon \rightarrow 0} \limsup_{n \rightarrow \infty} 1/n \log(N(X,\epsilon)) .
%$$

$$
h(f) = \lim_{\epsilon \rightarrow 0} \limsup_{n \rightarrow \infty}\frac{1}{n}\log(N(X,d_{f,n},\epsilon)).
$$

If $f$ is a Lipschitz self-map of a compact $m$-manifold with Lipschitz constant $K$, it's easy to see that 
$h(f) \le mK$.   The upper bound is attained in cases such as $x \mapsto K x$
acting on the torus $\R^n/\Z^n$, where $K$ is an integer.  On the other hand,
for a
continuous map $f$ that is not Lipschitz,  $h(f)$ need not be finite, even for
simple situations such as homeomorphisms of $S^2$ or continuous maps of intervals.

A differentiable map $f$ of an interval to itself is \emph{postcritically finite}
or \emph{critically finite} if the union of forward orbits of the critical points is finite.  In particular, the set of critical points for
$f$ must be finite.  

For a map $f$ that is not differentiable, we can define any point that is a local maximum or local minimum to be a turning point, or topological critical point. (Note that with this definition, not all smooth critical points are topological critical points.)  Let $c(f)$ denote the \emph{modality}, or number of turning points for such a map, and let $\Var(f)$ denote the total variation of $f$, \emph{i.e.} its arclength considered as a path.
For maps of the interval to itself with finitely many critical points, there are two simple ways to characterize the topological entropy:

\begin{theorem}[Misiurewicz-Szlenk, \cite{MisiurewiczSzlenk2},\cite{MisiurewiczSzlenk1}]\label{Misiurewicz-Szlenk}
For a continuous map $f$ of an interval to itself with finitely many turning points, the
topological entropy equals 
%$$ h(f) = \lim_{n \rightarrow \infty} \dfrac 1n\log( \Var(f^n)) = \lim_{n \rightarrow \infty} \dfrac 1n \log c(f^n) .
%$$
$$ h(f) = \lim_{n \rightarrow \infty}  \frac{1}{n}\log( \Var(f^n)) = \lim_{n \rightarrow \infty} \dfrac{1}{n} \log c(f^n).
$$

\end{theorem}

In particular, these are actual limits, not just limits of lim sups.  There
are good algorithms to actually compute the entropy \cite{MilnorThurstonKneading}.

\begin{figure}[htbp]
\centering
\includegraphics[width=3.5in]{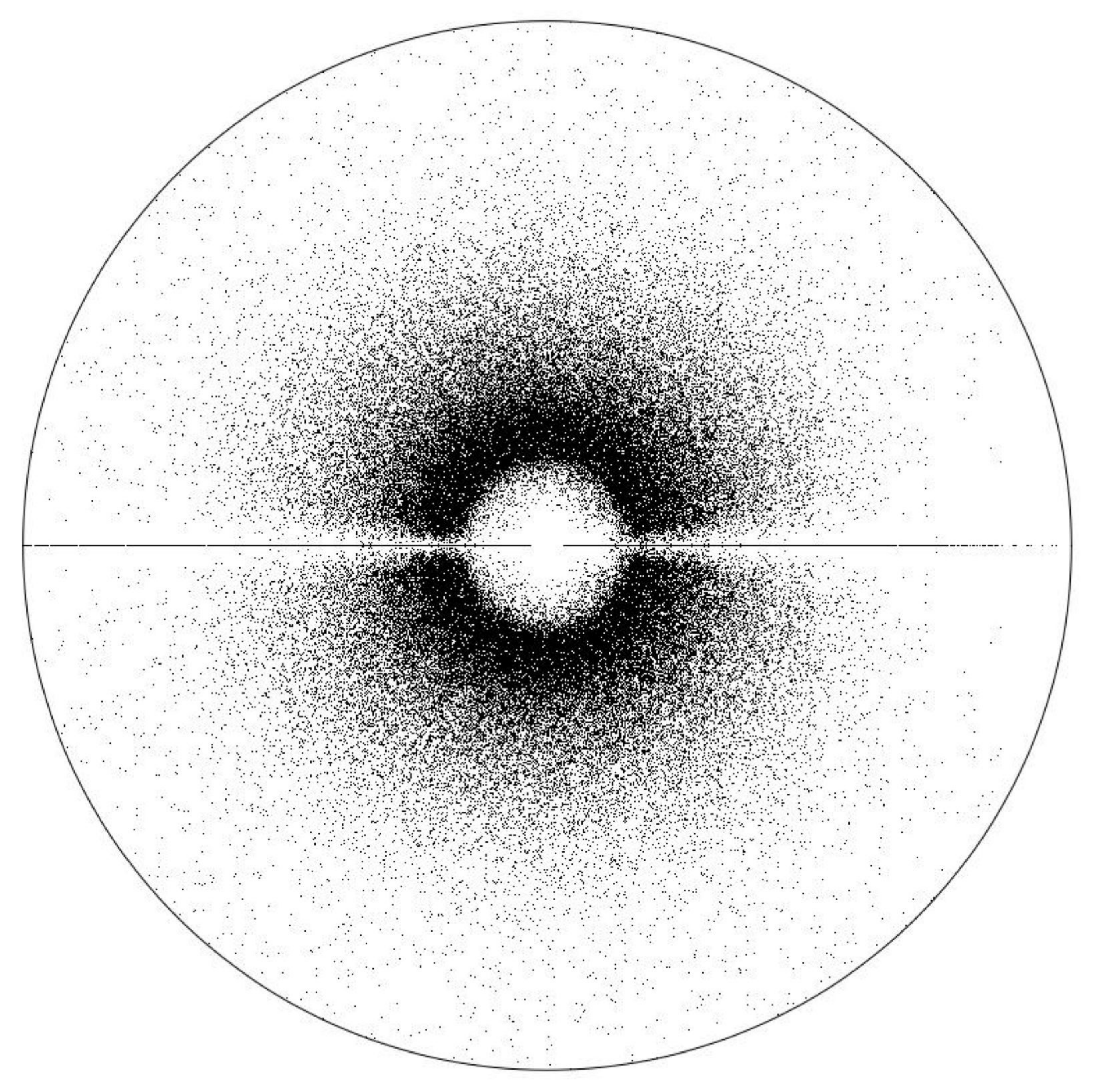}
\caption{ This plot shows the roots of the minimal polynomials for 5932 degree 21 Perron numbers, obtained by sampling 20,000 monic degree 21 polynomials
with integer coefficients between 5 and -5, and keeping those that have a root in $[1,2]$ larger than all other roots.}
\label{fig:PerronPoints}
\end{figure}

The first main goal of this paper is to characterize what values of entropy
can occur for postcritically finite maps:

\begin{theorem}\label{Main1} 
A positive real number $h$ is the topological entropy of a postcritically finite self-map of the unit interval if and only if $\exp(h)$ is
an algebraic integer that is at least as large as the absolute value of any
conjugate of $\exp(h)$. The map may be chosen to be a polynomial all of whose critical points are in $(0,1)$.
\end{theorem}

\medskip

Two maps $f_1, f_2: X \rightarrow X$ are \emph{conjugate} if there is a homeomorphism $g$ of $X$ conjugating one to the other, \emph{i.e.} $g \circ f_1 = f_2 \circ g$.  They are \emph{semiconjugate} if there is a map $g$ satisfying the condition that is continuous and
surjective, but not necessarily a homeomorphism.  Basically a semiconjugacy can collapse out certain kinds of subsidiary behavior of the dynamics.

Many phenomena of 1-dimensional dynamics are irrelevant for the study of entropy; there is 
a relatively simple family of non-smooth examples that has central importance, when $f$ is $\PL$ (piecewise-linear), and $|f'| = \lambda > 1$ is constant, wherever $f'$ exists.  We call such an $f$ a \emph{uniform expander}, or a \emph{uniform $\lambda$-expander} if we want to be more specific.  These maps are often called maps with \emph{constant slope}.
The importance of the uniform expanders is indicated by this theorem:

\begin{theorem}[\cite{MilnorThurstonKneading}]\label{PLEntropy}
Every continuous self-map $g$ of an interval with finitely many turning points is semi-conjugate to a uniform $\lambda$-expander $\PL(g)$ with the  the same topological entropy $\log(\lambda)$. If $g$ is postcritically finite, so is $\PL(g)$. (But if $\PL(g)$
is postcritically finite, it does not imply that $g$ is postcritically finite).
\end{theorem}

In \cite{ALM}, a more general version of theorem \ref{PLEntropy} is proven, which applies in more circumstances, including for instance piecewise continuous maps and self-maps of graphs.

In other words, theorem \ref{Main1} reduces to the study of expansion constants
for 1-dimensional uniform expanders.

In the case of a postcritically finite map $f$, a uniform expander model is
easily computed.  If necessary, first trim the domain interval until it maps to itself surjectively, by taking the intersection of its forward images. If this is a point, then the entropy is 0.  Otherwise, the two endpoints are either images of an endpoint, or images of a turning point. In this case,
conjugate by an affine transformation to make the  interval $[0,1]$.
%\sarah  the word be was removed from the phrase  "interval be $[0,1]$" . In the sentence above

If we now subdivide $[0,1]$ by cutting at the union of postcritical orbits (including the critical points) into intervals $J_i$,  each $J_i$ maps homeomorphically to a finite union of other subintervals. If there is a uniform expander $F$ with the same qualitative behavior, that is, having an isomorphic subdivision into subintervals each mapped homeomorphically to the corresponding union of other subintervals, then the lengths of the intervals satisfy a linear condition:
the sum of the
lengths of intervals $J_j$ hits is $\lambda$ times the length of $J_j$. In other words, the lengths of the intervals of the subdivision define a positive eigenvector for a non-negative matrix, with eigenvalue $\lambda$.

The Perron-Frobenius theorem gives necessary and sufficient conditions for this to exist. Here is some terminology:
a non-negative matrix
is \emph{ergodic} if the sum of its positive powers is strictly positive, and it is  \emph{mixing} if some power (and hence, all subsequent powers) is strictly positive.  The incidence matrix for $f$ is ergodic if and only if for each pair of intervals $J_i$ and $J_j$, some $f^n(J_i)$ contains $J_j$. The incidence matrix is mixing if and only if for each $J_i$, some $f^n(J_i)$ image covers all intervals.

The Perron-Frobenius theorem says that any non-negative matrix has at least one non-negative eigenvector with non-negative eigenvalue $\ge$ the absolute value of any other eigenvalue. If the matrix is ergodic, there is a unique strictly positive eigenvector; its eigenvalue is automatically [strictly] positive.  

From any
non-negative eigenvector for the incidence matrix, we can make a uniformly expanding model by subdividing the unit interval into subintervals whose lengths  equal the corresponding coordinate of the eigenvector normalized to have $L^1$ norm $=1$. We thus obtain a $PL$ map; its topological entropy is the log of the eigenvalue.

If all the entries of the matrix are integers, then its characteristic polynomial has integer coefficients, so it's an immediate corollary that the expansion constant for a postcritically finite uniform $\lambda$-expander is at least as large as the absolute value of its Galois conjugates.

In \cite{MilnorThurstonKneading} there is a more general formula (very quick on a computer) for a semiconjugacy to a uniform expander for a general map with finitely many critical points.

\medskip

Doug Lind proved a converse to the integer Perron-Frobenius theorem: 

\begin{theorem}[\cite{LindConversePerron}]\label{ConversePerron}
For any real algebraic integer $\lambda>0$ that is strictly larger than its Galois conjugates (in absolute value), there exists a non-negative integer matrix with some power that is strictly positive and has $\lambda$ as an eigenvalue.
\end{theorem}

\begin{figure}[htbp]
\centering
\includegraphics[width=4.6in]{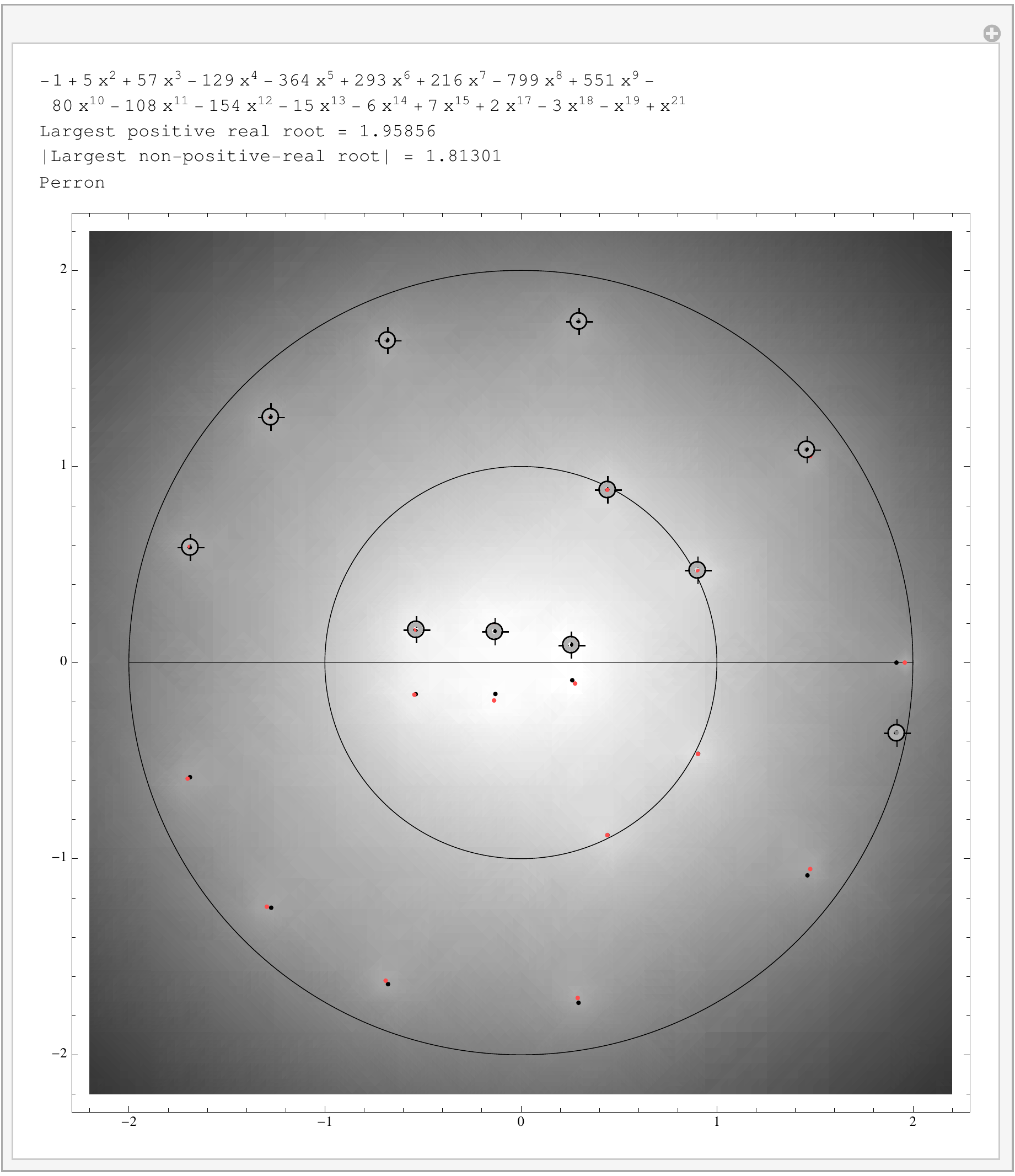}
\caption{This is an irreducible polynomial $P$ of degree 21 defining a Perron number, $\lambda = 1.95856$.  Although $\lambda$ is a unit $< 2$, the map $x \mapsto \lambda \abs x - 1$ cannot be
postcritically finite since $P$ has (several) roots not in (far away from) the region shown in figure \ref{fig:KneadingRoots}.
To construct $P$, first a monic real polynomial $P_0$ was defined by its roots with a mouse (the black dots). The coefficients of $P_0$ were rounded to define the integer polynomial $P$, whose roots are shown in red,
sometimes obscuring the roots of $P$. The constant term was made small by balancing roots inside and outside the unit circle. When roots for $P_0$ are chosen so that those away from the unit circle are spaced well apart (relative to the  sizes of coefficients), they tend to be fairly stable under rounding. The shading, proportional to $\log(1 + \abs P )$, is a guide to stability: clicking new roots into darker areas is typically stabilizing.}
\label{fig:ChosenPerron}
\end{figure}

We will prove Lind's theorem on our way to other results, in section
\ref{sec:ProofMain1a}.

In view of Lind's theorem together with the Perron-Frobenius theorem, these numbers are called \emph{Perron} numbers. A real algebraic integer $\lambda$ 
that satisfies the weak inequality $\lambda \ge |\lambda^\alpha|$ where $\alpha$
ranges over the Galois group of $\lambda$ is a \emph{weak Perron number}.

There are two important (and better known) special cases of Perron numbers.
A positive real algebraic integer is a \emph{Pisot number}, or \emph{Pisot-Vijayaraghavan number} or \emph{PV number} if all its Galois conjugates are
in the open unit disk.  Since the product of all Galois conjugates is a nonzero
integer, $\lambda$ is bigger than its conjugates. It turns out that
the set of Pisot numbers is a closed subset of $\R$.
If  $\lambda > 1$ is a real algebraic integer that has at least one conjugate
on the unit circle, and all conjugates are in the closed unit disk, then $\lambda$ is called a
%\sarah removed 's' from "unit circles"... in this sentence
\emph{Salem number}.  Since the complex conjugate of a point on the unit circle
is its inverse, every element of the Galois conjugacy class of a point on the unit circle is also Galois conjugate to its inverse and in particular
$\lambda$ is Galois conjugate to $1/\lambda$.  Therefore $1/\lambda$ is the only Galois conjugate of $\lambda$ in the open unit
disk, since the inverse of any other conjugate in the open unit disk would give
another conjugate of $\lambda$ outside the unit disk.

\emph{Note}: The size of the matrix in theorem \ref{ConversePerron} can be larger than the degree of $\lambda$. One way to see this (suggested by Doug Lind) is by considering Perron numbers with negative trace, like the positive real root of $p(t)=t^3+3t^2-15t-46$; $p$ cannot be the characteristic polynomial for a $3\times 3$ matrix with non-negative entries as its trace would be $-3$.  In fact, for any integer $n > 0$ there are cubic Perron numbers $\lambda$ that are not eigenvalues for non-negative matrices smaller than $n \times n$.

%\emph{Note}: The dimension of the matrix in Lind's theorem is not bounded by the degree of $\lambda$:  for any integer $n > 0$ there are in fact cubic Perron numbers $\lambda$ that are not eigenvalues for non-negative matrices smaller than $n \times n$.

The proof of theorem \ref{Main1} uses techniques motivated by Doug Lind's methods.

It is also interesting to investigate what happens for postcritically finite maps subject to a bound on the number of turning points, in particular, a single turning point (\emph{i.e.} for quadratic polynomials).  The situation is very different.  Figure \ref{fig:KneadingRoots} shows the Galois conjugates of $\exp(h)$
for postcritically finite real quadratic maps. This is a path-connected set, with much structure visible. Most Perron numbers between $1$ and $2$ do not have
roots in this set.  For example, figure \ref{fig:PerronPoints} shows a sampling
of degree 21 polynomials with coefficients between -5 and 5 that happen to define a Perron number between 1 and 2.
(out of 20000 random polynomials, 5937 fit the condition).   These however are
not random Perron numbers of degree 21: more typically, many of the coefficients
are much larger.  Figure \ref{fig:ChosenPerron} shows a degree 21 example constructed by hand, first specifying a collection of real points and pairs of complex conjugate
points in the disk of radius 2, expanding the monic polynomial with those points as roots, and rounding the coefficients to the nearest integers.  With care in spacing and placement of roots, the integer polynomial has roots near the given choices. (When the points away from the unit circle cluster too much, their positions become unstable with respect to rounding).

\medskip
More generally, if $\Gamma$ is a finite graph
and $f: \Gamma \rightarrow \Gamma$
is a continuous map which is an embedding when restricted to any edge, we will
say that $f$ is \emph{postcritically finite} if the forward orbit of every vertex is finite.  If $f$ has the additional property that for all $k$, $f^k$ restricted to any edge is an immersion (it is an embedding on a sufficiently
small neighborhood of any point), then $f$ is a \emph{train track map}.

Entropy for graph maps behaves similarly to entropy for intervals:
\begin{theorem}[Alsed{\'a}-Llibre-Misiurewicz \cite{ALM}]
Let $f: \Gamma \to \Gamma$ be a continuous self-map of a finite graph which has
finitely many exceptional points $x(f)$ where $f$ is not a local homeomorphism.
Then
\[h(f) = \lim_{n \to \infty} \frac{1}{n} \log\Var(f^n) \]
If $f$ is a degree $d$ covering map $S^1 \to S^1$, this yields $h(f) = \log(d)$; otherwise,
the entropy also satisfies
\[ h(f) =\lim_{n \to \infty} \frac{1}{n} \log( x(f^n)).
\]
\end{theorem}

A self-map of a graph that is a homotopy equivalence defines an outer automorphism of its fundamental group, that is, an automorphism  up to conjugacy (since we're not specifying a base point that must be preserved).

Handel
and Bestvina \cite{BestvinaHandel} showed that for any outer automorphism that is irreducible in the sense that no free factor is preserved up to conjugacy, there is a graph $\Gamma$ and a train track map of $\Gamma$ to itself representing the outer automorphism. They also developed a theory of relative train tracks that addresses outer automorphisms that are reducible.  Train track theory is a powerful tool, parallel in many ways to pseudo-Anosov theory for self-homotopy-equivalences of surfaces.  Algebraically, you can look at the 
action of an outer automorphism on conjugacy classes, represented by cyclically
reduced words in a free group. The lengths of images of cyclically reduced words
have a limiting exponential growth rate.  There is a well-understood situation when train track maps can
have conjugacy classes that are fixed under an automorphism: for instance,
any automorphism of the free group on $\{a,b\}$ fixes the conjugacy class
$[a,b]$.  Apart from these, the lengths of the sequence of images of any conjugacy classes under iterates of
a train track map have exponent of growth equal to the topological entropy of the train track map, as measured in any generating set.

Peter Brinkmann wrote a very handy java application {\tt Xtrain} that implements the Bestvina-Handel algorithm,
%\sarah the word automorphism was changed to algorithm
 {\tt http://math.sci.ccny.cuny.edu/pages?name=XTrain}. I used this 
program extensively to work out and check examples for the
next theorem, which is the second main goal of this paper:

\begin{theorem}\label{Main2} A positive real number $h$ is the topological
entropy for an ergodic train track representative of an outer automorphism of a free group if and only if its expansion constant $\exp(h)$ is an algebraic integer that is at least as large as the absolute value of any conjugate of $\exp(h)$.
\end{theorem}

\emph{Note}: Even though automorphisms are invertible, the expansion constant
need not be an algebraic unit.

The relationship between the expansion constant of an automorphism and the expansion constant of its inverse is mysterious, but there is one special case
where it's possible to control the expansion constant for both an automorphism $\phi$
and its inverse $\phi^{-1}$.   An automorphism $\phi$ is \emph{positive}
with respect to a set $G$ of free generators if $\phi$ of any generator is a
positive word in the generators, that is, it preserves the semigroup they generate. This implies that $\phi$ is a train track map of the bouquet of circles defined by $G$. 

\begin{defn}
A linear transformation $A$ is \emph{bipositive} with respect to a basis $B$ if
$B$ can be expressed as the disjoint union $B = P \cup N$ such that 
$A$ is non-negative with respect to $B$, and its inverse is non-negative with respect to the basis $P \cup -N$.
An automorphism $\phi$ is \emph{bipositive} if it is positive
with respect to a set $G$ of free generators, and its inverse is positive with respect to a set of generators obtained by replacing some subset of elements of $G$ by their inverses. 
\end{defn}
\noindent{\emph{Example.}} Let
\[
A=\left(\begin{array}{cc}  1   & 1 \\  1  &2  \end{array}\right),\quad N=\left\{\left(\begin{array}{c}  1   \\  0 \end{array}\right)\right\},\quad\text{and}\quad P=\left\{\left(\begin{array}{c}  0   \\  1 \end{array}\right)\right\}.
\]
The matrix $A$ is bipositive with respect to $B=P\cup N$. 

\medskip

At one point, I hoped that the criteria in the following theorem would 
characterize all pairs of expansion constants for a free group automorphism and its inverse. 
%\sarah removed an 's' from the phrase "expansions constants" in the sentence above. 
This turned out to be false (Theorem \ref{Main2}), but
the characterization of such pairs in this special case is still interesting:

\begin{theorem}\label{Main3} A pair $(\lambda_1, \lambda_2)$ of positive real numbers is the pair of
expansion constants for $\phi$ and $\phi^{-1}$, where $\phi$ is bipositive, if and only if it is the pair of positive eigenvalues for a bipositive element of $GL(n,\Z)$ for some $n$,
if and only if $\lambda_1 $ and $\lambda_2 $
are real algebraic units such that the Galois conjugates of $\lambda_1$ and $\lambda_2^{-1}$ are contained in the closed annulus
$\lambda_2^{-1} \le |z| \le \lambda_1$.
\end{theorem}

Theorem \ref{Main3} does not extend in an immediate way to the general case. Classification of the set of pairs of expansion constants that can occur in general remains mysterious. As already noted,
these expansion constants need not be units.  Moreover,
there are examples of train track maps where the Galois conjugates
of  $\lambda_1$ and $\lambda_2^{-1}$ are *not* contained in the annulus
$\lambda_2^{-1} \le |z| \le \lambda_1$.  It is consistent with what I currently
know that \emph{every} pair of weak Perron numbers greater than 1 is the pair of expansion constants for $\phi$ and $\phi^{-1}$. The proof of theorem \ref{Main3} will be sketched in section \ref{sec:Doubletracks}.

\section{Special case: Pisot numbers}\label{sec:Pisot}

The special case that $\lambda$ is a Pisot number has a particularly easy theory, so we will look at that first. 

It is easy to see that the topological entropy of a map $f: I \rightarrow I$
with $d-1$ topological critical points can be at most $\log(d)$: each point has at most $d$ preimages,
so the total variation of $f^n$ is at most $d^n$.

\begin{theorem}
For any integer $d > 1$ and any Pisot number $\lambda \le d$,  there is a postcritically finite map $f: I \rightarrow I$ of degree $d$ (that is, having $d-1$ critical points) with entropy $\log(\lambda)$. 

In fact, when $\lambda$ is Pisot, every $\lambda$-uniformly expanding map whose critical points are in $\Q(\lambda)$ is postcritically finite.
\end{theorem}

\begin{figure}
\centering
\includegraphics[width=3.5in]{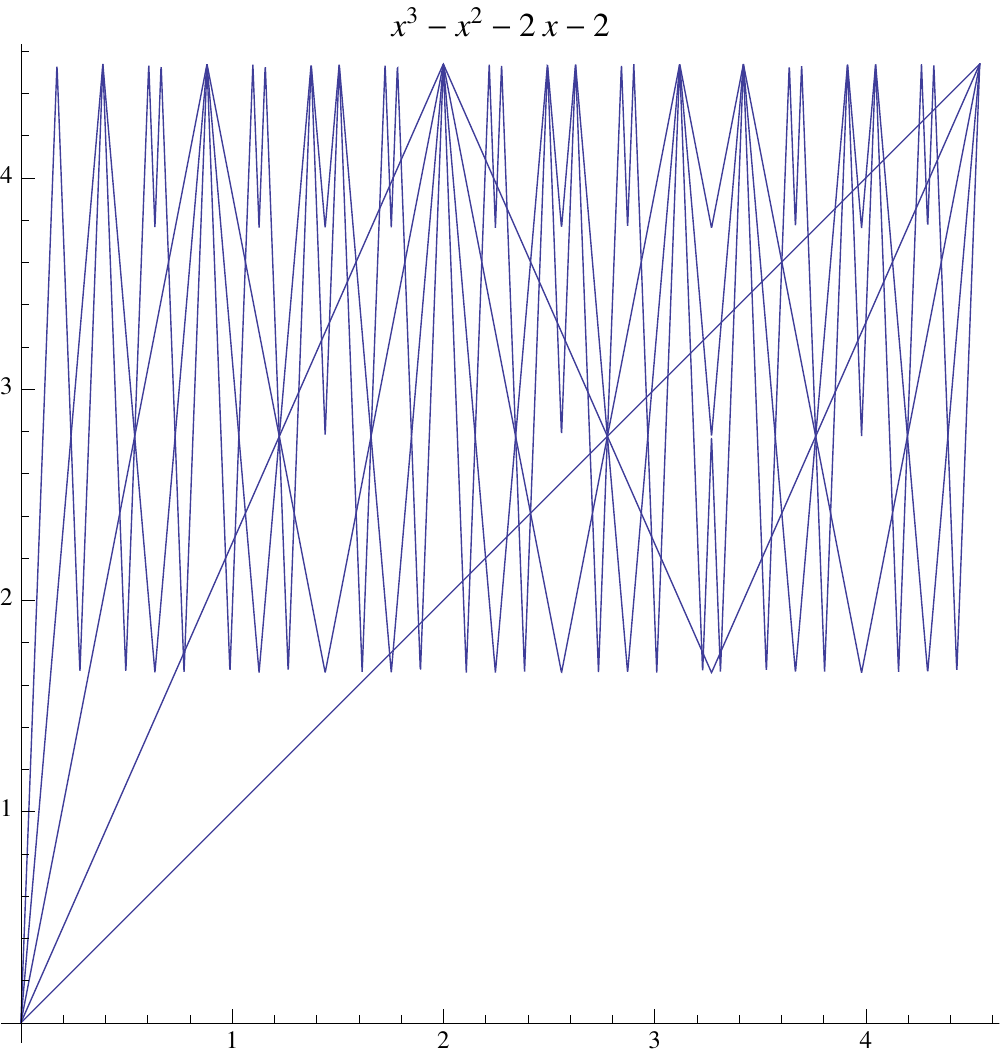}
\caption{This is the graph of the first 4 iterates of the function described in
the proof, for the Pisot number satisfying $x^3-x^2-2x-2 = 0$.  The Pisot root is $\lambda = 2.2695308$, and the other roots are $-0.63476542+0.69160123i$ and its conjugate, of modulus $0.938743$. The first critical point, 2, maps to the endpoint $2\lambda$, which is fixed. The other critical point, $1+\lambda$, maps to a fixed point on the third iterate.}
\end{figure}

\begin{proof}
One way to construct pure $\lambda$-expanders is to create their graphs by folding.  Start with the graph of the linear function $x \mapsto \lambda x$
on the unit interval. Now reflect the portion of the graph above the line $y=1$
through that line.  Reflect the portion of the new function that is below the
line $y=0$ through that line.  Continue, until the entire graph is folded into
the strip $0 \le y \le 1$.

This results in a function that may have fewer than $d-1$ critical points,
so repeatedly reflect segments of the graph through horizontal segments $y=\alpha$ at
heights $\alpha \in \Q(\lambda)$ until the function has
the desired number $d-1$ of critical points.  

For convenience, rescale by some integer $n$ to clear all denominators, so that all critical points are algebraic integers.
  Therefore, the postcritical orbits are contained in
$\Z[\lambda]$.

\begin{figure}[htbp]
\centering
\includegraphics[width=5in]{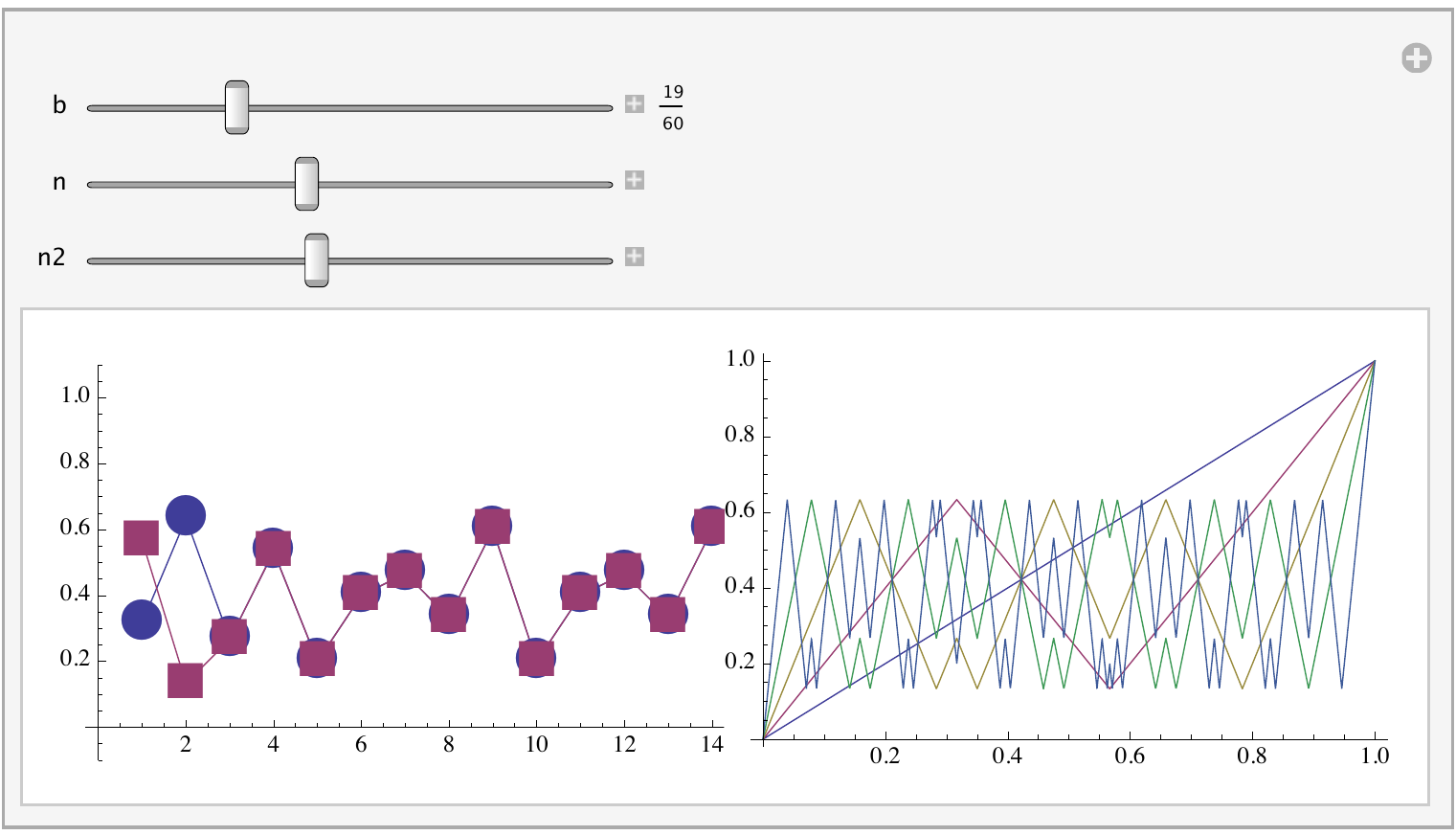}
\caption{The diagram at left shows the postcritical orbits for the Pisot construction where $\lambda= 2$, $d = 3$, with the critical points
chosen as $19/60$ and $17/30$. On the right is the plot of the first four iterates of this piecewise linear function of entropy $\log(2)$.}
\label{fig:lambda2d3}
\end{figure}

For each Galois conjugate $\lambda^\alpha$ of $\lambda$, there is an embedding of $\Z[\lambda]$ in $\C$.  In each such embedding, the postcritical orbits remain bounded:  the action of $f_{\lambda^\alpha}$ on any point is some composition of functions of the form $f_{\lambda^\alpha,i}(x) = \pm \lambda^\alpha(x) + a_i$, which act as contractions, so there is a compact subset $K \subset \C$ that the Galois conjugates of every piece, $f_{\lambda^\alpha,i}$  takes inside itself.

By construction, the orbit of the critical points under $f_\lambda$ is also bounded, since $f_\lambda$ is a map of an interval to itself.

For any fixed bound $B$, there are only a finite number of algebraic integers of $\Q(\lambda)$ satisfying the bound $\abs {\lambda^\alpha} < B$, since in the
embedding into the product  of real embeddings and a selection of one from each pair of complex conjugate embeddings, the algebraic integers in
$\Q(\lambda)$ form a lattice. The postcritical set is contained in such a set, so it is finite. 
\end{proof}

This phenomenon is closely related to why decimal representations of rational numbers are eventually periodic. There is a theory of $\beta$-expansions, similar to decimal expansions but with
$\beta$ a real number; when $\beta$ is Pisot, the digits of the $\beta$ expansion of any element of $\Q(\lambda)$ are eventually periodic (by an almost
identical proof). This kind of argument appears in \cite{bertrand}, \cite{gelfond}, and \cite{schmidt}. 
\medskip

The details of construction of $f_\lambda$ above are not important. \emph{Any} $\lambda$-expander
whose critical points are in $\Q(\lambda)$ will do.  As long as $\lambda \ne d > 2$, there are infinitely many.  To make the proof work as phrased, rescale the unit interval to clear all denominators, so that all critical points become algebraic integers in $\Q(\lambda)$.

Even in the case that $\lambda$ is an integer less than $d$, this gives countably many different examples provided $d > 2$.  For instance, figure \ref{fig:lambda2d3} shows the critical point orbits when $\lambda=2$, $d=3$ and the critical points are chosen
as $19/60$ and $17/30$. On the fourth iterate, they settle into a single
periodic orbit of period 4.  There is a unique cubic polynomial, up to affine conjugacy, having the same order structure for the postcritical orbits, with entropy
$\log(2)$.  

In general, there is a non-empty convex $d-2$-dimensional space of $\lambda$-uniform-expanders for every $1 < \lambda < d$. If $\lambda$ is Pisot,
then postcritically finite examples are dense in this set. 

There are many Pisot numbers: for any real algebraic number $\alpha$, it is easy to see that there are
infinitely many Pisot numbers in $\Q(\alpha)$:  the intersection of the lattice of algebraic integers with a cylinder centered around any line corresponding to an embedding of $\Q(\alpha)$ in $\R$ consists of Pisot numbers, except for those in a closed ball containing the origin.
However, in the geometric sense, Pisot numbers are rare:  in \cite{salem}, Salem proved that the set of Pisot numbers is a countable closed subset of $\R$, making use of a theorem of Pisot that a real number $x$ is Pisot if and only if sequence of minimum differences of $x^n$ from the nearest integer is square-summable. The golden ratio is the smallest accumulation point of Pisot numbers, and the \emph{plastic number} $1.3247 \dots$,
 a root of $x^3 - x - 1$ is the smallest Pisot number. 
 
 It is elementary and well-known that postcritically finite maps are dense among uniform expanders, but the construction above raises a question 
 that does not seem obvious for $d > 2$:
 
 \begin{question}
 For fixed $d$, is there a dense set of numbers $1 < \lambda < d$ for which the set of postcritically finite maps is dense among $\lambda$-uniform expanders?
 For which $\lambda$ are there infinitely many non-affinely equivalent postcritically finite maps? For which $\lambda$ are postcritically finite maps dense?
 \end{question}
 
One way to get infinite families of postcritically finite maps with the same $\lambda$ is to take dynamical extensions of maps with fewer critical points
(\emph{c.f. section \ref{sec:DynamicalExtensions}}), taking care only to introduce new critical points that map to positions whose forward orbit is finite.  But this construction cannot work when $\lambda > d-1$.

\begin{figure}[htpb]
\centering
\includegraphics[width=5in]{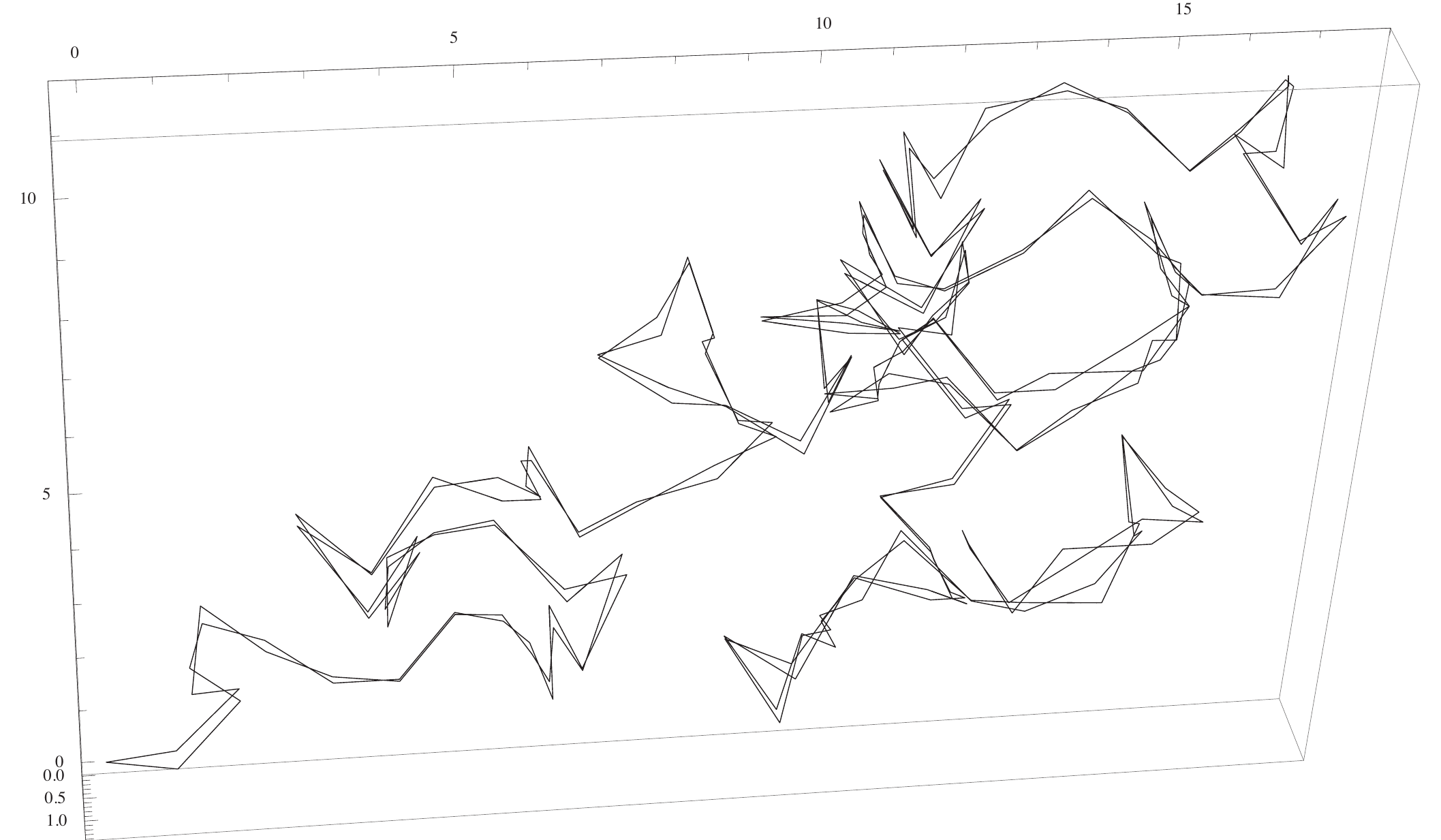}
\caption{The tent map $x \mapsto \lambda \abs x - 1$ for the degree 6 Salem number $\lambda = 1.4012683679\dots$ that satisfies $\lambda^6-\lambda^4-\lambda^3-\lambda^2+1 = 0$ is postcritically finite, with the critical point having period 270, quite large compared to the degree of $\lambda$.  This figure shows the absolute value of the projection
to the $\lambda$-line (the thin direction) as well as to the two complex places of $\lambda$.  The trajectory resembles a random walk in the plane. Since random walks in $\E^2$ are recurrent (they have probability 1 of visiting any set of positive measure infinitely often), one would expect it to eventually return. It does, but as random walks often do (the mean return time is $\infty$), it takes a long time.
}
\label{fig:SalemPeriod270}
\end{figure}

Salem numbers are closely related to Pisot numbers: some people conjecture that the union of Salem numbers and Pisot numbers is a closed subset of $\R$. However, the construction that worked for Pisot numbers is inadequate
for Salem numbers. The Galois conjugates of the
linear pieces of $f_\lambda$ are isometries of $\C$ for any conjugate on the unit circle. Let $n$ be the degree of the Salem number, so for each linear piece of $f_\lambda$ there are $n/2-1$ Galois conjugate complex isometries, one for each complex place.
If we take
the product of these isometries over all complex places of $\Q(\lambda)$ we get
an action by isometries on $\C^{n/2 - 1}$, where the first derivative of the action of each linear piece is $\pm U_\lambda$, where $U_\lambda$ is 
a unitary transformation.  In the unitary group, the orbit is dense on an $n/2-1$ torus, acting as an irrational rotation of the quotient of the torus by $\pm I$ (thus factoring out the complication of the variable sign of $f_\lambda)$.
The action of the sequence of linear functions on the $(n/2-1)$-tuple of moduli 
appears to behave like a random walk in $\R^{(n-2)/2}$.
Sometimes they are periodic, sometimes with fairly large periods, but Brownian motion  in dimension bigger than 2 is not recurrent, and a few experiments for $n \ge 8$ indicate that they typically drift slowly toward infinity, and
thus are not postcritically finite. However, there could be some reason (opaque to me) why they might not act randomly in the long run, and they could eventually cycle. At least it seems hard to prove  any particular example is not postcritically finite.

\begin{figure}[htbp]
\centering
\includegraphics[width=3in]{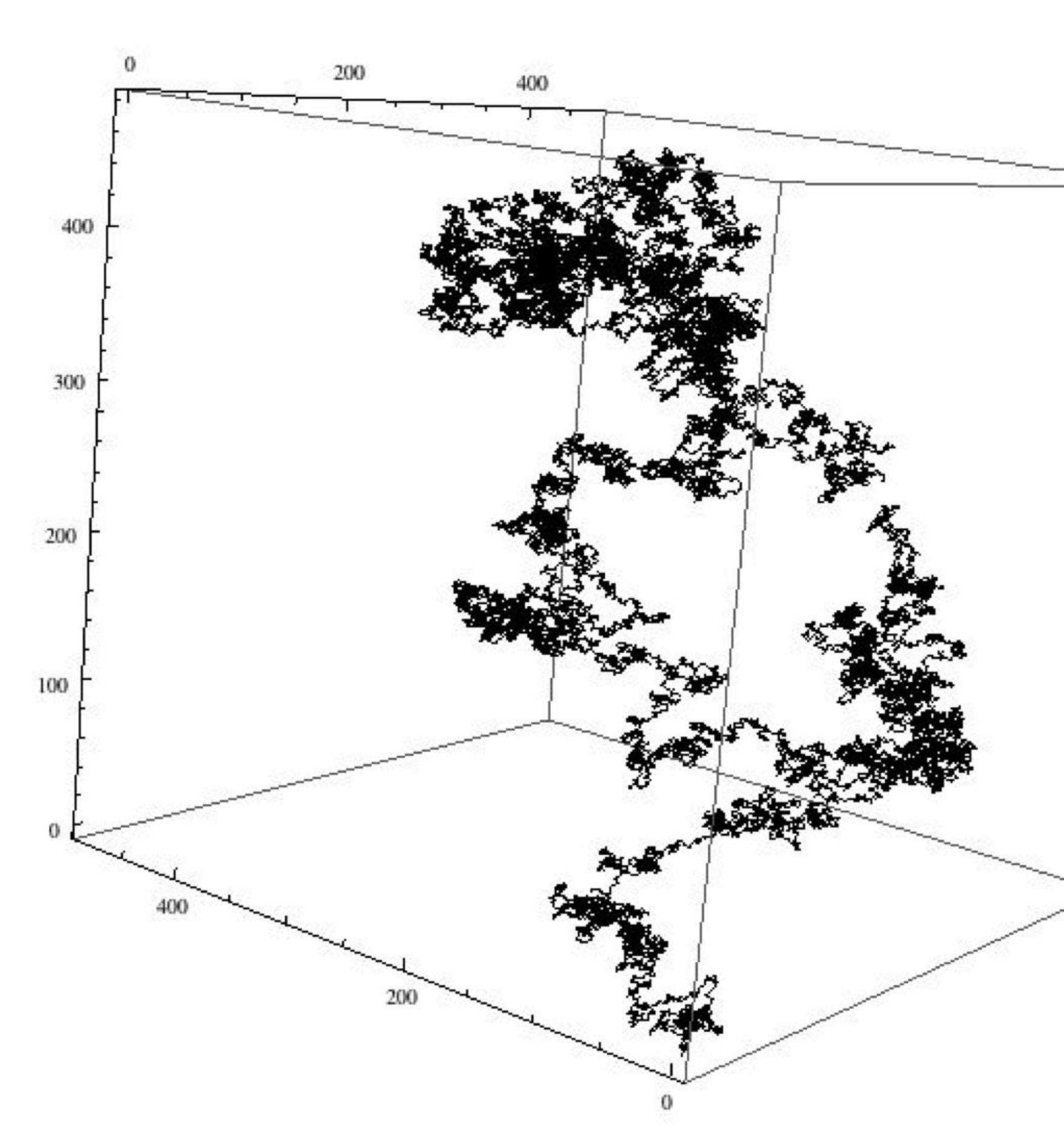}
\caption{This diagram is a 3-dimensional projection of the first 200,000 iterates of the critical point for $x \mapsto \lambda|x| - 1$ where 
$\lambda = 1.17628\dots$ is Lehmer's constant with minimal polynomial 
$1 + x - x^3 - x^4 - x^5 - x^6 - x^7 + x^9 + x^{10}$. There are 4 complex places; this shows a 3-dimensional projection for the quadruple of absolute values.  It resembles a Brownian path in 3 dimensions. The map is postcritically finite if and only if the path closes. Its values are always algebraic integers, so if it comes sufficiently close to the start it's fairly likely to close. The quadruple of radii determines a 4-torus in $\C^4 \times \R^2$, with the current lattice point a bounded distance from that torus. When the radii are on the order of 200, this bounded neighborhood has volume on the order of $(2\pi*200)^4$, with roughly  $10^{12}$ lattice points, so the chances of returning seem remote.}
\label{fig:LehmersWalk}
\end{figure}

For example, for the Salem number 1.7220838 satisfying $x^4-x^3-x^2+x-1=0$ of degree 4, the critical point of the tent map is periodic of period 5. For
the Salem number $1.401268367939\dots$ satisfying $t^6-t^4-t^3-t^2+1 = 0$, the
critical point has period 270.  Note that its square is also a Salem number, for
which the period is $135 = 270/2$.  

One of the most famous Salem numbers is the Lehmer constant, defined by the polynomial
$1 + x - x^3 - x^4 - x^5 - x^6 - x^7 + x^9 + x^{10}=0$. The single root outside the unit circle is
$1.17628\dots$. This is the smallest known Salem number,
and in fact the smallest known Mahler measure for any algebraic integer. (Mahler
measure is the product of the absolute value of all Galois conjugates outside the unit circle.)  Figure \ref{fig:LehmersWalk} is a diagram of the first 200,000 elements of the critical point in $x \mapsto \lambda |x| - 1$, projected
from $\C^4 \times \R^2$ to the quadruple of radii in the complex factors, and from there a projection into 3 dimensions. The map is postcritically finite if and only if the path closes. Its values are always algebraic integers, so if it comes sufficiently close to the start it's fairly likely to close. The quadruple of radii determines a 4-torus in $\C^4 \times \R^2$, with the current lattice point a bounded distance from that torus.   When the radii are on the order of 200, as in this case, this bounded neighborhood has volume on the order of $(2\pi*200)^4$, with roughly  $10^{12}$ lattice points, so the chances of looping appear remote unless the path wanders close to the origin, where the tori are smaller.  For comparison, the variance of a random walk with stepsize 1
in
 $\R^n$ equals the number of steps, so for a random walk of length 200,000, the standard deviation is $\sqrt{200,000} \approx 447$; projected from 8 dimensions to 3, the standard deviation would be $\approx 274$,  in line with the
picture.  An experiment with a selection of small Salem numbers of moderate degree $> 6$ showed similar results, with none of them exhibiting a cycle within $500,000$ iterates. 

This discussion is related to the ideas surrounding the $\beta$-transformation $T_\beta:[0,1]\to [0,1]$ given by $T_\beta:x\mapsto \beta x$, where $\beta>1$. The number $\beta$ is a {\em{beta number}} if $1$ has finite orbit under $T_\beta$. In \cite{schmidt}, Klaus Schmidt showed that every Pisot number is a beta number (proved independently in \cite{bertrand}). In \cite{boyd}, David Boyd proved that if $\beta$ is a Salem number of degree $4$, then it is a beta number. In \cite{boyd2}, David Boyd presents heuristic arguments based on random walks that almost every Salem number of degree $6$ should be beta. Thanks to Doug Lind for bringing this to my attention.

\section{Constructing interval maps: First Steps}\label{sec:ProofMain1a}

For a Perron number that is not Pisot, the situation is much more delicate. To
develop a strategy, first we'll discuss incidence matrices. There are two reasonable versions for an incidence matrix that are transposes
of each other. We will use the version whose columns each represent the image of a subinterval of the domain, and whose rows represent
a subinterval in the range, so that each entry counts how many times the subinterval that indexes its column crosses the subinterval that indexes its row.

Suppose the unit interval $I$ is subdivided into $n$ subintervals $J_1, \dots, J_n$ (in order).  Let $V = \left \{ 0 = v_0, v_1, \dots, v_n = 1\right \}$ be
the vertex set. Every function $g: V \rightarrow V$ that takes adjacent vertices to distinct vertices extends to a postcritically finite piecewise linear map.
In this way, a large but finite and specialized set of $n\times n$ matrices can be realized as incidence matrices.  Incidence matrices are easy to recognize. Each column consists of a consecutive sequence of $1$'s, and is otherwise $0$. There is a matching between the ends of the blocks of consecutive $1$'s in adjacent columns, with every column except the first and last having one end of the block of $1$'s matched to the left and the other end of the block of $1$'s matched to the right. 

Here is a generalization of this concept. Consider a map $f: I \to I$ with finitely many critical points such that $f(V) = V$ and $f$ of the critical set is contained in $V$, that is, $V$ contains all critical values.  Under these
conditions, an extended incidence matrix is still defined for $f$, whose entries $a_{ij}$ count how many times $f(J_j)$ crosses interval $J_i$.

\begin{proposition}\label{IncidenceMatrixCharacterization}
An $n \times n$ non-negative integer matrix $A$ is an extended incidence matrix
if and only if
\begin{enumerate}
\item{The nonzero entries in each column form a consecutive block, and}
\item{There is a map $\phi:\{0,1,\dots, n\} \rightarrow \{0,1,\dots, n\}$
such that in column $i$, the entries in rows greater than $\phi(i-1)$ and
not greater than $\phi(i)$, and no other entries, are odd.}
\end{enumerate}
\end{proposition}

\begin{figure}[htpb]
\centering
\hbox{
\includegraphics[width=2.6in]{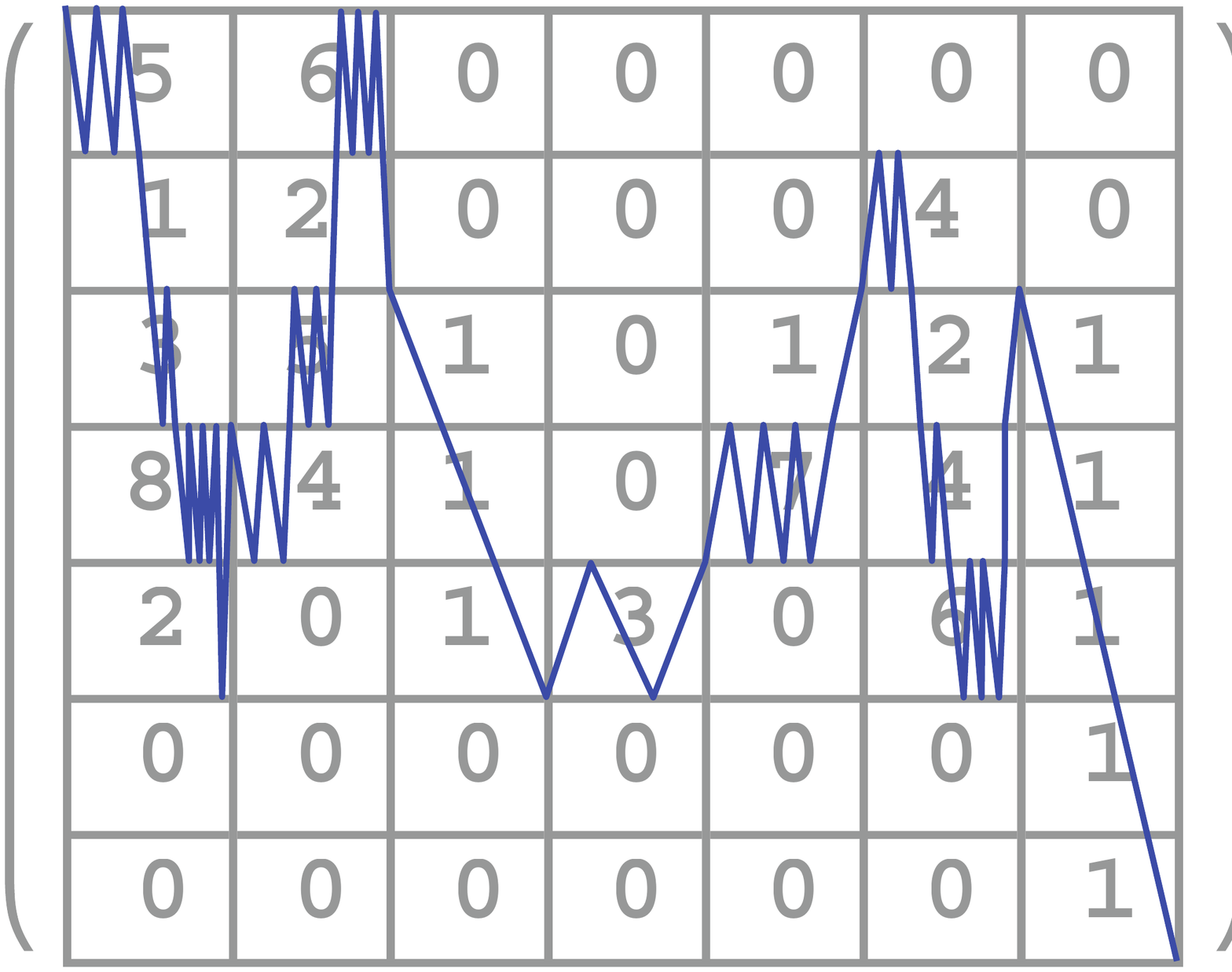}
\includegraphics[width=2.2in]{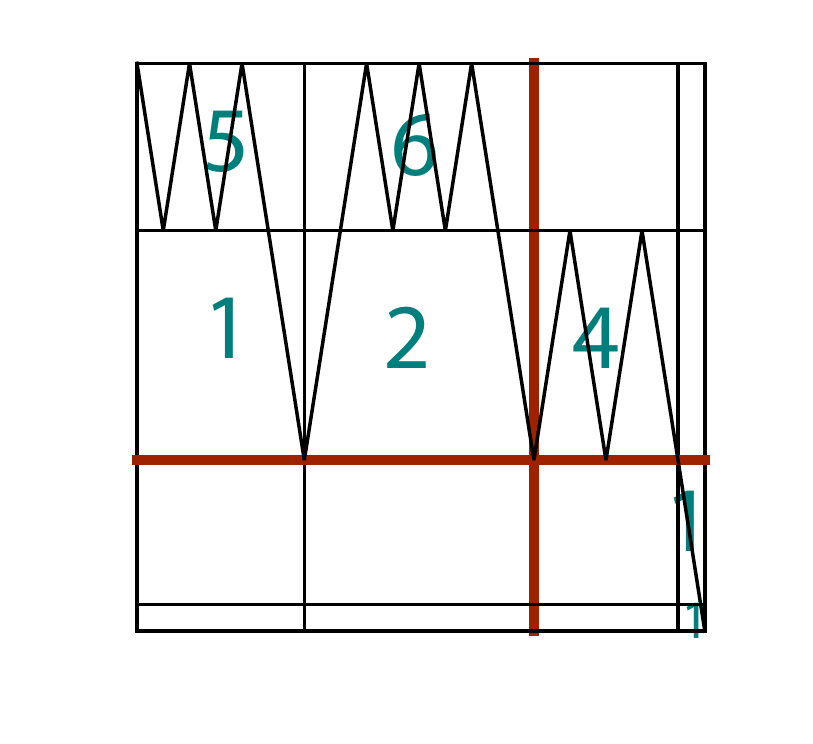}
}
\caption{
The matrix $M$ on the left
satisfies the necessary and sufficient conditions of proposition
\ref{IncidenceMatrixCharacterization} to be an incidence matrix for a 
selfmap of the interval: in each column, the positive entries in each column form a connected 
block as do the odd entries, and the odd blocks in the columns can be matched together end to end, from the left column to the right, to form a chain.
The blue line defines a PL function realizing the given matrix, since it
crosses each square the specified number of times. If you reflect the line across the top edge of the matrix, it matches the usual convention for
drawing a graph (since the convention for
matrices that the row numbers increase going downward is opposite the convention 
for graphs of functions).
On the right, the widths of rows and columns of the matrix have been adjusted
in proportion to the positive eigenvector for the transpose $M^t$. 
The matrix $M^t$ is not ergodic, entries 3,4,5 of the positive eigenvector are 0, and
the shaded blocks at left have collapsed to the red lines at right. Now
the graph can be drawn with constant absolute slope.
}
\label{fig:IncidenceMatrix}
\end{figure}
\begin{proof}
The necessity of the conditions is easy.  The map $\phi$ represents the map $f$ restricted to $V$.
The image of any  interval $J_i$  is necessarily the union of a consecutive block of intervals. The subinterval between the images of the first and last endpoints, $[f(v_{i-1}), f(v_i)]$ (which could be a degenerate interval) is traversed an odd number of times, and the rest of the image is traversed an even number of times.

Sufficiency of the conditions is also easy. To map $J_i$, start from $v_{\phi(i-1)}$, go to the lowest vertex in the image, zigzag across the lowest interval until its degree
is used up, then the next lowest, etc. until you get to $v_{\phi(i)}$, at which point proceed to the highest vertex in the image and work back. 
\end{proof}

\emph{Note}: Sufficiency can also
be reduced to the familiar condition that a graph admits a Hamiltonian path from vertex $a$ to vertex $b$ if and only if it is connected and either $a = b$ and all vertices have even valence, or $a$ and $b$ have odd valence and all other vertices have even valence.  For each $i$, apply this to the graph $\Gamma_i$
that has $a_{ij}$ edges connecting $v_{j-1}$ to $v_j$.  The entire map $f$ is
really a Hamiltonian path in the graphs $\Gamma_i$ connected in a chain by
joining  vertex $\phi(i)$ of $\Gamma_i$ to that of $\Gamma_{i+1}$, followed by the natural projection to $[0,1]$.

\begin{proposition}
The topological entropy of any map with extended incidence matrix $A$ is the log
of the largest positive eigenvalue of $A$.
\end{proposition}
\begin{proof}
The total variation of $f^n$ is a positive linear combination of matrix entries of $A^n$ (if all intervals have equal length, it is their sum). By \ref{Misiurewicz-Szlenk}, $h(f)$ is the exponent of growth of total variation, so this equals the log of the largest positive eigenvalue of $A$.
\end{proof}

\medskip

Suppose we are given a (strict) Perron number $\lambda$.  Our strategy is to first construct a strictly positive extended incidence matrix that has positive eigenvalue $\lambda^N$, for a large power $N$ of $\lambda$. We will promote this to an example with entropy $\lambda$ by implanting it as the return map replacing a periodic cycle of a map with entropy less than $\lambda$.

Afterwards, we will deal with additional issues involving questions of mixing and weak Perron numbers.

\medskip
Given $\lambda$, let $O_\lambda$ be the ring of algebraic integers in the field
$\mathbb Q(\lambda)$, and let $V_\lambda$ be the real vector space
$V_\lambda = \mathbb Q(\lambda) \otimes_\mathbb{Q} \mathbb R$. Another way to think of it is that $V_\lambda$ is the product of the real and complex places of $\mathbb Q(\lambda)$, that is, the product of a copy of $\mathbb R$ for each real root of the minimal polynomial for $\lambda$ and a copy of $\mathbb C$ for each pair of complex conjugate roots. The ring operations of $\mathbb Q(\lambda)$ extend continuously to $V_\lambda$, but division is discontinuous
where the projection to any of the places is 0.  Yet another way to think of
$V_\lambda$ is in terms of the companion matrix $C_\lambda$ for the minimal polynomial $P_\lambda$ of $\lambda$. We can identify $\mathbb Q(\lambda)$ with the set of all polynomials in $C_\lambda$ with rational coefficients, and $V_\lambda$ with the vector space on which the companion matrix acts.  The real and complex places of $\mathbb Q(\lambda)$ correspond to the minimal invariant subspaces of $C_\lambda$. When $P_\lambda$ is factored into polynomials that are irreducible over $\mathbb R$, the terms are linear and positive quadratic; the subspaces are in one-to-one correspondence with these factors.

We are assuming that $\lambda$ is a Perron number, so for multiplication of $V_\lambda$ by $\lambda$, the $\lambda$-eigenvector is dominant.  If we consider the convex cone $K_\lambda$ consisting of all points where the projection to the 
$\lambda$ eigenspace is larger than the projection to any of the other invariant subspaces, then $\lambda * K_\lambda$ is contained in the interior of $K_\lambda$ except at the origin.

\begin{proposition}
There is a rational polyhedral convex cone $KR_\lambda$ contained in $K_\lambda$ and containing $\lambda* K_\lambda$.
\end{proposition}
\begin{proof} It is easiest to think of this projectively, in $\mathbb P(V_\lambda)$.
The projective  image $\proj(K_\lambda)$ of $K_\lambda$ is a convex set (specifically, a product of intervals and disks), with the image of $\proj(\lambda* K_\lambda)$ contained in its interior. Since rational points $\proj(\Q(\lambda))$ are dense in $\proj(V_\lambda)$, we can
readily find a set of rational points in the interior of $\proj (K_\lambda)$ whose convex hull contains $\proj(\lambda* K_\lambda$). This gives us the desired
rational polyhedral convex cone.
\end{proof}

Let $S_\lambda$ be the additive semigroup $O_\lambda \cap KR_\lambda \setminus \{0\}$.

\begin{proposition}
$S_\lambda$ is finitely generated as a semigroup.
\end{proposition}
\begin{proof}
This is a standard fact.  Here's how to see it using elementary topology:  
We can complete $S_\lambda$  by adding on the set of projective limits $\proj(KR_\lambda)$.  The completion is compact.  It's easy to see that the set of closures $U_s$ of the ideals $s+S_\lambda$ for $s \in S_\lambda$ form a basis for this topology.  By compactness, the cover by basis elements has a finite subcover $U_{s_1}, \dots, U_{s_k}$.  For any such cover, the set $G = \left \{ s_1, \dots, s_k\right\}$ form a generating set: given any element, write it as $s_i + s$, and continue until the remainder term $s$ is 0.
\end{proof}

\begin{proof}[Proof of Converse Perron-Frobenius \ref{ConversePerron}]
Equipped with this picture, it is now easy to prove the converse Perron-Frobenius theorem of Lind.  Consider the
free abelian group $\Z^G$ on the set $G$ of semigroup generators. The
positive semigroup in $\Z^G$ maps surjectively to $S_\lambda$. Call this map $p$. We can lift the
action of multiplication by $\lambda$ to an endomorphism of the positive
semigroup, by sending each generator $g$ to an arbitrary element of $p^{-1} (\lambda * p(g))$. In coordinate form, this is described by a non-negative integer  matrix.  In $S_\lambda$, 
projection to the $\lambda$-space is a dual eigenvector, that is, a linear functional multiplied by $\lambda$ under the transformation. It is strictly positive on $S_\lambda$. Therefore, the pullback of this function is a strictly
positive $\lambda$-eigenvector of the transpose matrix, proving Lind's theorem.
\end{proof}

\medskip
Note that the minimum size of a generating set might be much larger than  the dimension of $V_\lambda$.  For instance, if $\lambda$ is a quadratic algebraic integer, $V_\lambda$ is the plane, and $KR_\lambda$ could be bounded by any pair
of rational rays that make slightly less than a $45^\circ$ angle to the $\lambda$ eigenvector.  Minimal generating sets can be determined using continued fraction expansions of the slopes; they can be arbitrarily large. 

As we shall presently see, it can happen in higher dimensions that the minimum size of a generating set for $S_\lambda$ can be 
very large no matter how we choose a semigroup $S_\lambda$ on which
multiplication by $\lambda$ acts as an endomorphism.

\section{Second Step: Constructing a map for $\lambda^N$}\label{sec:ProofMain1b}

Now we need to address the special requirements for an incidence matrix for a
map having a finite invariant set as the set of critical values.  Given any
Perron number $\lambda$ of degree $d$, we will construct such a matrix for some power, probably large, of $\lambda$.  From the previous section, \ref{sec:ProofMain1a},
we assume we have a set $G$ of semigroup generators for a semigroup in $S_\lambda \subset V_\lambda = \Q(\lambda) \otimes_\Q \R $ invariant under multiplication by $\lambda$.

Choose a finite sequence of generators, including each generator at least once, such that the partial sums of the sequence contain all $2^d$ mod 2 congruence classes, that is, the partial sums map surjectively to
$O_\lambda / (2 O_\lambda)$. If necessary, adjoin additional elements to the sequence so that the sum $T$ of the entire sequence, mod 2, is 0.  

The action of $\lambda$ (by multiplication) on the projective completion of $S_\lambda$ has a unique attracting fixed point. For any $s \in S_\lambda$, the closure of $s + S_\lambda$ contains the fixed point, so there is some power $N$  such that $\lambda^N * S_\lambda \subset 3 T+S_\lambda$. 

Now mark off an interval of length $T$ into segments whose lengths are given
by the chosen sequence $g_1, \dots, g_k$ of generators (in the embedding of $O_\lambda$ in $\R$ where $\lambda$ goes to $\lambda$).  We'll construct a $\lambda^N$-expander map by
induction, going from left to right, starting with $0 \rightarrow 0$. There is
some point $q$ in this subdivision that has the same value mod $2O_\lambda$ as
$\lambda^N * g_1$. Since $\lambda^N * g_1$ can be expressed as $3 T$ plus a sum
of generators, it can also be expressed as $q + 2T$ plus a sum of generators.
Since $(\lambda^N * g_1 - q - 2T)$ is in $S_\lambda$ and congruent to zero mod 2,
it is divisible by 2 in $S$: it can be expressed as  $2 \alpha$ where $\alpha \in S_\lambda$.  We can write $2\alpha$ as a linear combination of generators
with even coefficients, so we can write $\lambda^N * g_1$ as a strictly positive
sum of generators, where each generator between $0$ and $q$ occurs an odd number of times, and each other generator occurs an even number of times.

We can continue in exactly the same way for $\lambda^N$ times each of the points in the subdivision. We first pick where each point $g_i$ goes  based on the
congruence class of $\lambda^N *  g_i$ mod 2, then express the difference as an
even and strictly positive linear combination of all generators in the sequence.
Finally we end with $T$ going to either $0$ or $T$, as we choose. The incidence
matrix satisfies the conditions of \ref{IncidenceMatrixCharacterization}, so
we have constructed a $\lambda^N$ uniform expander.

\section{Powers and Roots: Completion of proof of Theorem \ref{Main1}} \label{sec:ProofMain1c}

Given a Perron number $\lambda$, we'll construct a map of $S^1$ to itself that is a $\lambda$ uniform expander, because the construction is a little nicer for $S^1$.  From the preceding section, for some $N$ we construct a $\lambda^N$ uniform expanding map $f_{\lambda^N}$ of an interval that takes each endpoint to itself.  Let $\rho$ be any rotation of the circle of order $N$.  The circle can be subdivided
into $N$ intervals that are cyclically permuted by the rotation.  Define a metric on the circle so that in this cyclic order, the $N$ intervals have length
$1, \lambda, \lambda^2, \dots, \lambda^{N-1}$.  Now define $g_\lambda: S^1 \rightarrow S^1$ by mapping each of the first $N-1$ intervals affinely to the next, and mapping the last interval to the first using $f_{\lambda^N}$ with
affine adjustments in the domain and range to send the last interval exactly to the first.  Since the first interval has length $1/\lambda^{N-1}$ times the last, with this affine adjustment $f_{\lambda^N}$ also expands uniformly by $\lambda$.

In the case of the circle, we can easily modify the construction of $g_\lambda$ to make the incidence matrix mixing: this will happen if we change any small piece of $f_{\lambda^N}$ to stray into a neighboring interval and back, when it gets to the upper endpoint. If the rotation is chosen as a rotation by $1/N$ and if the first generator is chosen to be a "small" element in $S_\lambda$, these intervals are small, and straying is probably possible with ease, but in any case, by taking $N$ to be a somewhat higher power, it can be readily guaranteed.

\begin{figure}[htbp]
\centering
\includegraphics[width=4in]{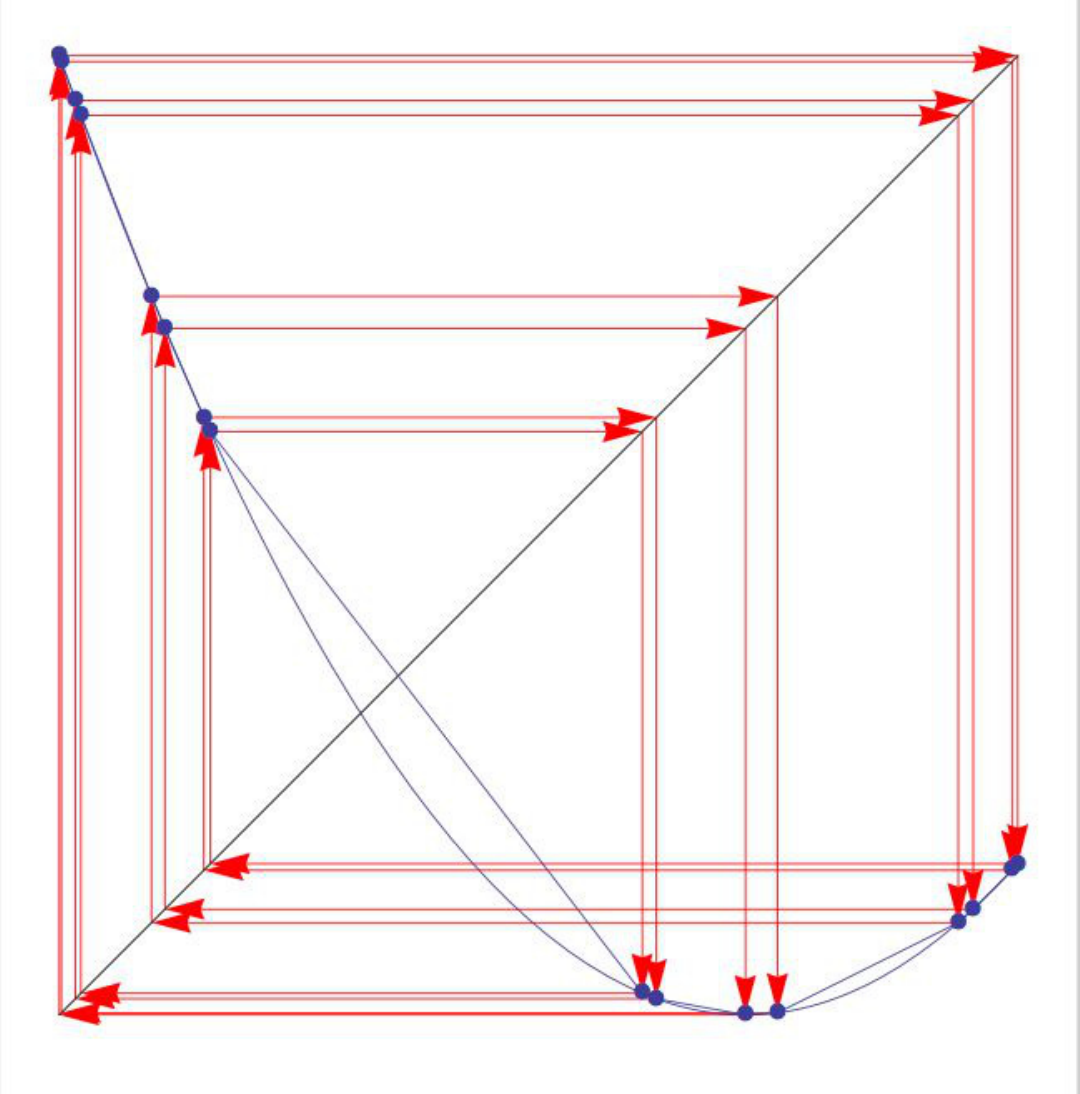}
\caption{This is the graph of a quadratic map in the initial period-doubling
cascade, with critical point of period 16. This, and all other quadratic maps
in this family, have entropy 0.  If the critical orbit is blown up and replaced
with a sequence of intervals where the return map has is a $\lambda^{16}$-uniform expander,
then the resulting map is semiconjugate to a $\lambda$-uniform expander.
}
\label{fig:PeriodDoubling}
\end{figure}

In the case of the unit interval, we need a substitute for a rotation of the circle. 
For this, we can use the well-known period-doubling cascade for quadratic self-maps of an interval (see figure \ref{fig:PeriodDoubling} for an illustrative example).  In this  period-doubling family, there is a quadratic map $q_n$ with entropy 0 in which the critical point has period $2^n$. We can blow up
the forward and backward orbit of the critical point, replacing each point $x$
in the orbit by
a small interval $I_x$ of any length $l_x$  such that  the set of lengths is summable.  Extend the map to these intervals by affine homeomorphisms, with the
exception of the interval for the critical point; for that, we can use any affine map that takes both endpoints to 0.

Since the entropy of the quadratic map is 0, the number of critical points of $q_n^k$ grows subexponentially in $k$, so if we assign length $\lambda^{-k}$ to  each interval for a point that is critical for $q_n^k$ but not for $q_n^{k-1}$, the set of lengths is
summable.  After blowing up the full orbit of the critical point in this way, we
can define a pseudo-metric on the interval where the length of an interval is the
measure of its intersection with the blown-up orbit; everything else collapses to measure 0.

By the preceding section, we can find an $f_{\lambda^N}$ of the form $N = 2^n$
that maps the unit interval to itself, taking both endpoints to $0$. Implant this, using affine adjustments in the domain and the range, for the map from the interval for the critical point of $q_n$ to its image.  The image of the critical interval has length $\lambda^{2^n - 1}$ since the original critical point had period $2^n$, so the implanted map is a uniform $\lambda$-expander.

We now have a map which is a local homeomorphism with Radon derivative $\lambda$ in the complement of the critical interval, where it is a uniform $\lambda$-expander.  Therefore, the entire map is a piecewise-linear uniform $\lambda$-expander, with entropy $\log(\lambda)$.  This completes the proof of theorem \ref{Main1} for strong Perron numbers.

\medskip
Now we'll address weak Perron numbers.
\begin{proposition}
A positive real number $\lambda$ is a weak Perron number if and only if some power of $\lambda$ is a [strong] Perron number.
\end{proposition}
\begin{proof}
In one direction this is pretty obvious: the $n$th roots of any algebraic integer are algebraic integers, and their ratios to each other are $n$th roots of unity. Thus the positive real $n$th root of a Perron number is a weak Perron number.

In the other direction, suppose $\lambda$ is a weak Perron number. Let
$B$ be the set of Galois conjugates of $\lambda$ with absolute value $\lambda$, and let $b$ be their product. This product $b$ is a real number equal to $\lambda^{\#(B)}$.  The Galois conjugates of $b$ are
products of $\#(B)$ Galois conjugates of $\lambda$; for any subset of conjugates of this cardinality other than $B$, the product is strictly smaller, so $b$ is a Perron number. Since $\lambda^n = b$, the other elements of $B$ also satisfy this equation, so their ratio to $\lambda$ is a root of unity.
\end{proof}

Now given a weak Perron $\lambda$, we first find a power $k$ so that $\lambda^k$ is a Perron number.  In the family of degree 2 uniform expanders, functions with periodic critical point are dense, and functions with critical point having period a multiple of $k$ are also dense.  Choose such a function $g$ whose entropy is less than $\lambda$, where the critical point has period a sufficiently high multiple $k n$ of $k$. Blow up the full (backwards and forward) orbit of the critical point, and implant a map of the form 
$f_{\lambda^{k n}}$, as constructed in the preceding section \ref{sec:ProofMain1b}, in the interval replacing the critical point. Since the growth rate of critical points for powers of $g$ is less than the growth rate of powers of $\lambda$, we can construct a metric just as before that is uniformly expanded by $\lambda$.
\medskip

Note that the $\lambda$-uniform expanders we have constructed are very far from mixing. This is of course impossible if $\lambda$ is only a weak Perron number, by the Perron-Frobenius theorem, but for a Perron number $\lambda$ it is tempting to try to generalize the straying technique that worked earlier.

There are two difficulties.  The first is an essential problem:
\begin{proposition}
For  \emph{any} self-map $f$ of the interval with entropy in
the interval $(0, \log (\sqrt 2))$, there are two disjoint subintervals that are interchanged by the map, and in particular, $f$ is not mixing.
\end{proposition}
\begin{proof}
For a map $f:I \mapsto I$, let $D(f)$ be the set of points that have more than one preimage. Note that $D(f) \subset D(f^2)$ and also $f(D(f)) \subset D(f^2)$;
in fact, $D(f^2) = D(f) \cup f(D(f))$.
Let $a = \inf (D(f^2))$, and $b = \sup(D(f^2))$.
Then $[a,b]$ is mapped into itself, and has the same entropy as $f$.

\end{proof}

In particular, they never mix.
By induction, if the entropy is less than $2^{-n}\log 2$, there are $2^n$ disjoint intervals cyclically permuted by the map.

The second problem is that the piecewise-linear map we constructed above has
coefficients in $\Q(\lambda)$ (either because the formulas from kneading theory for the  infinite sums of intervals are rational functions of $\lambda$, or because they are determined by linear functions with coefficients in $O_\lambda$),
but it is not obvious how to get coordinates to be in $O_\lambda$.  Maybe it's possible to analyze and control the algebra, but if so it's beyond the scope of this paper.

\section{Maps of Asterisks}\label{sec:Asterisks}

Selfmaps of graphs, including in particular
Hubbard trees, give another interesting collection of postcritically finite maps of graphs.   For use later in constructing automorphisms of free groups, we will look at a special case, the \emph{asterisks $\ast$}. An \emph{$n$-pointed asterisk} is the cone on a set of $n$ points,
which we'll refer to as the \emph{tips} of the asterisk.

\begin{theorem}\label{AsteriskMaps}
For  every Perron number $\lambda$ there is a postcritically finite
$\lambda$ uniform expanding self-map $f$ of some asterisk such that 
\begin{itemize}
\item{$f$ fixes the center vertex, and}
\item{each edge maps to an edge-path in a way that every edge is the first element of the image edge-path of some edge}
\item{the incidence matrix for $f$ is mixing.}
\end{itemize}
\end{theorem}
A map is \emph{topologically transitive} if there are dense [forward] orbits under $f$.  

\begin{remark}
Any expanding self-map of a tree can be promoted to a self-map of a Hubbard tree by adding extra information as to a planar embedding, and choosing a branched covering map for the planar neighborhood of each vertex that acts on its link in the given way. Each such promotion is the Hubbard tree for a unique polynomial up to affine automorphism.
\end{remark}

\begin{proof}
This is similar to the proof for self-maps of the interval. In principal
it is easier because there is no order information to worry about, but we will
use a very similar method.

An incidence matrix for
a self-map of an asterisk of the given form is a non-negative integer matrix that has exactly one odd entry in each row and each column.  Given any such
matrix, we can construct a corresponding self-map of an asterisk by permuting the tips according to the matrix mod 2, which is a permutation matrix, and running each edge out and back various edges to generate the even part of the matrix.

As before, find a subsemigroup $S_\lambda$ of $O_\lambda$ that excludes 0 and is invariant 
under multiplication by $\lambda$. 

Let $G$ be a finite set of generators 
for $S_\lambda$ that maps surjectively to $O_\lambda / (2 O_\lambda)$.
Let $T = \sum_{g \in G} g$.

Let $n$ be an integer such that  $\lambda^n- 1$ is
congruent to 0 mod 2, and for each $g \in G$,
$\lambda^n * g $
is contained in $g +2(T + \lambda * g) + S_\lambda$.

Now make an asterisk whose points are indexed by $G \times \{1,2,\dots, n-1 \}$.
Map the edge to point $(g, i)$ homeomorphically to the edge to
point $(g, i+1)$ when $i < n-2$.  

The final set of tips $(g,n-1)$ will map to tips $(g,1)$.
For each $g$, $\lambda^n * g - g$ is congruent to 0 mod $2O_\lambda$, and  so can be expressed as a strictly positive even linear combination of
$G \cup \{\lambda*  g\}$.  

Use these linear combinations to construct an asterisk map.  Each edge to the first set of tips maps homeomorphically to a new edge for the first $n-1$ iterates. On the $n$th iterate, it maps to a path that makes at least one round trip to all the first layer points as well as one second layer point, finally ending back where it started.  We can easily arrange the order of traversal so
that every edge is represented in the first segment of one of these edge paths.

If the asterisk is given a metric where edge $(g,i)$ has length equal to the
value of $\lambda^i g$ in the standard embedding of $O_\lambda$ in $\R$ where $\lambda$ is the Perron number, this map is a $\lambda$-uniform expander.  It is mixing: by the $n$th iterate, the image of an edge with index $(g,i)$ contains
all generators of the form $(*,i)$; by the $2n$th iterate, the image of an edge 
contains all generators of the form $(*,i)$ and $(*, i+1)$. After $n^2$ iterates, the edge maps surjectively to the entire asterisk.
\end{proof}

This construction was intended to avoid the need for complicated conditions and 
bookkeeping.
It's clear that a more careful construction could produce suitable
asterisk maps for a typical Perron number $\lambda$ that are much smaller (but still might be quite large). 

\section{Entropy in bounded degree}\label{sec:EntropyinBounded}

The constructions for maps of given entropy have been very uneconomical with the complexity of the maps. First, there is a potentially dramatic (but sometimes unavoidable) blowup in going from a Perron number $\lambda$ to a finite set of generators for a subsemigroup $S_\lambda$ of algebraic integers invariant by multiplication by $\lambda$. Even then, there is another possibly large blowup in finding a power of $\lambda$ such that $\lambda^N * S_\lambda$ is sufficiently deep inside $S_\lambda$ to guarantee an easy construction of a suitable incidence matrix.   In other words:  unlike Pisot numbers, a typical Perron number is probably unlikely to be the growth constant for a typical postcritically finite map of an interval to itself.

\begin{figure}[htpb]
\centering
\includegraphics[width=4in]{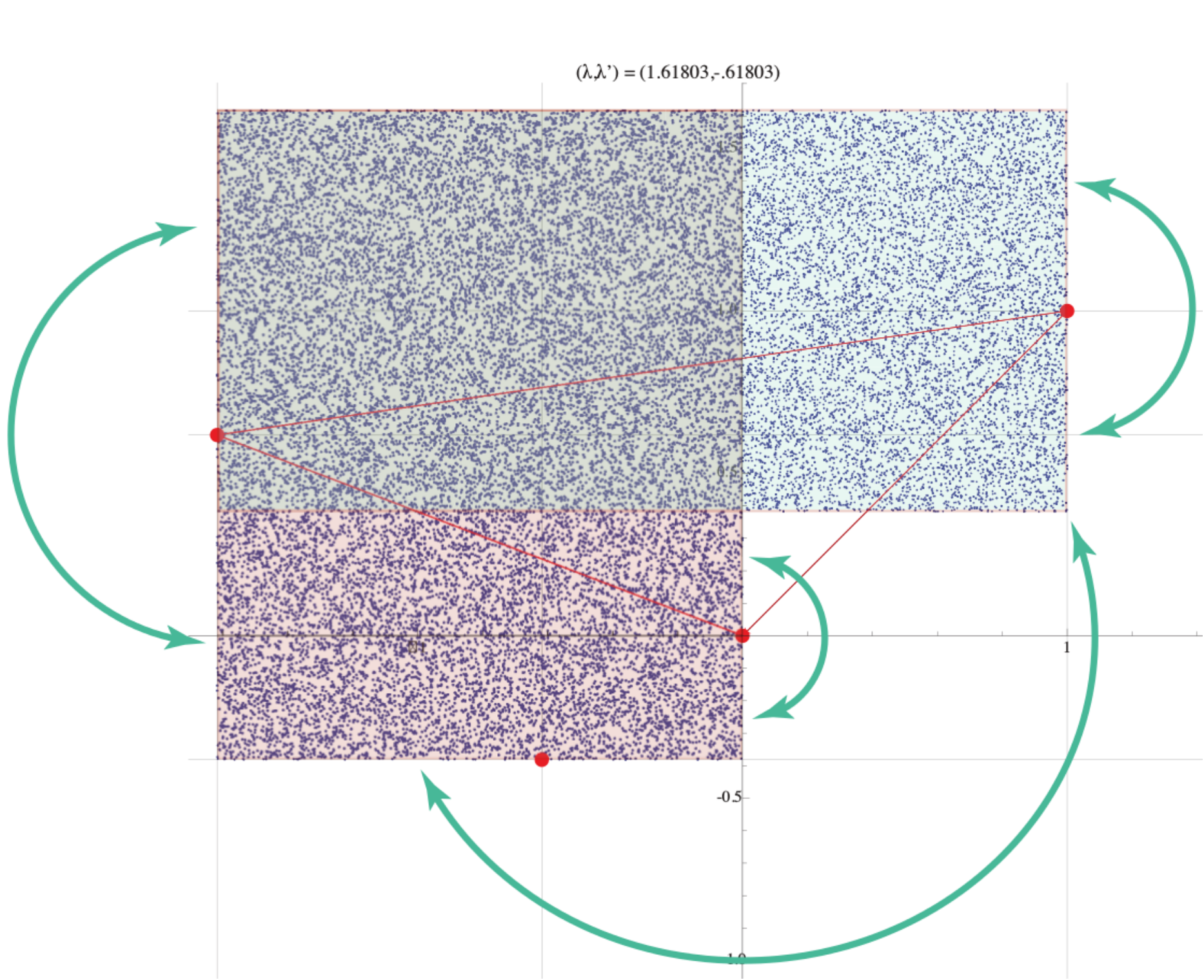}
\caption{This diagram shows $L(f^{1+\sigma})$ where
$\lambda = 1.61803\dots$, the golden ratio, is a
Pisot number, and $\sigma(\lambda) = -1/\lambda$. It was drawn by taking a random point near the origin, first iterating it 15,000 times and then plotting the next 30,000 images. The dynamics is hyperbolic and ergodic, so almost any point would give a very similar picture.
The critical point is periodic of period 3, and the 
limit set is a finite union of rectangles with a critical point on
each vertical side. The dynamics reflects the  light blue rectangle on
the upper right in a horizontal line, and arranges it as the pink rectangle
on the lower left. The big rectangle formed by the two stacked rectangles at left is reflected in a vertical line, stretched horizontally and squeezed vertically, and arranged as the two side-by-side rectangles on the top. Although the dynamics is discontinuous, the sides of the figure can be identified, as indicated (partially) by the green arrows to form a tetrahedron to make it continuous: see figure
\ref{fig:FoldedFibonacciTetrahedron}.
}
\label{fig:FoldedFibonacci}
\end{figure}
\begin{figure}
\centering
\includegraphics[width=3in]{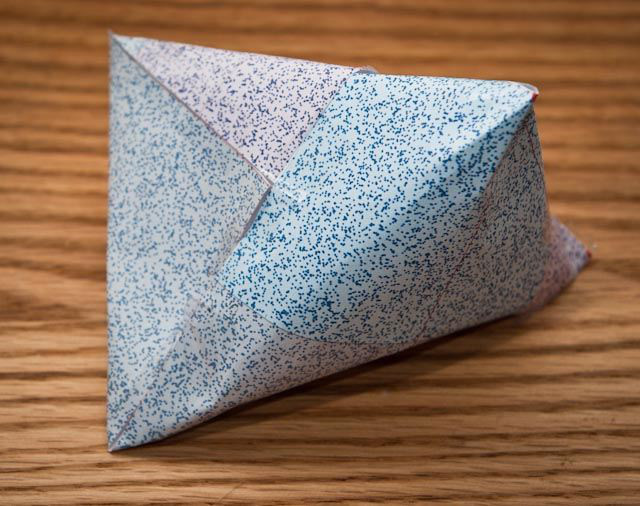}
\caption{ This is the diagram from figure \ref{fig:FoldedFibonacci} taped together into a tetrahedron, where the dynamics acts continuously (but reverses orientation). The map $f^{1+\sigma}$ acts on the tetrahedron as
an Anosov diffeomorphism. There are coordinates in which it becomes
the Fibonacci recursion  $(s,t) \mapsto (t, s+t)$, modulo a (2,2,2,2) symmetry group
generated by $180^\circ$ rotations about lattice points.}
\label{fig:FoldedFibonacciTetrahedron}
\end{figure}

In some sense, expansion constants for bounded degree systems are
almost Pisot: most of their Galois conjugates don't seem to wander very far outside the unit circle.  Figure \ref{fig:KneadingRoots} illustrates this point: most of the Galois conjugates of $\exp(h(f))$ for postcritically finite quadratic maps are in or near the unit circle. Since their minimal polynomials
have constant term
 $ \pm 1$ or $\pm 2$, the inside and outside roots are approximately balanced. If they don't wander outside the circle, they can't wander very far inside the circle, and most roots are near the unit circle.  In contrast, 
  figure \ref{fig:PerronPoints} illustrates that Perron numbers less
  than $2$ can have roots spread in the disk of their radius.

To get some understanding of what's going on, we'll translate questions about the distribution of roots into questions about dynamics of a semigroup of affine maps, elaborating on the point of view taken in the discussion of Pisot numbers 
in section \ref{sec:Pisot}.    We'll consider the sets of piecewise linear uniform expander functions $F(d, \epsilon)$ that take $\partial I $ onto $\partial I$, where $d>1$ is the number of
intervals on which the function is linear and $\epsilon = \pm 1$ determines
the sequence of slopes. $\epsilon (-1)^i \lambda$ is the sequence of slopes in
subintervals $i=0$ through $i=d$.
The constant terms for the linear function $f_i$ in
the first and last interval are determined by the condition that $f(\partial I) = \partial I$, implying that $f_0(0) = 0$ if $\epsilon = 1$ and $f_0(0) = 1$ when $\epsilon = -1$, with a similar equation for the last interval.  In all other subintervals, the constant term is a free variable $C_i$ subject to linear 
inequalities.  That is, the $i$th critical point $c_i$, determined by
$\epsilon (-1)^{i-1} \lambda c_i + C_{i-1} = \epsilon (-1)^i \lambda c_i + C_i$,
must be inside the unit interval,
so we have the inequalities

\[ 0 \le  \epsilon (-1)^{i-1} \frac {C_i - C_{i-1}} {2 \lambda} \le 1 .
\]

Now for any other field embedding $\sigma:
\Q(\lambda, C_1, \dots, C_{i-1})$ in $\C$, we can look at the collection
of image functions $f_i^\sigma$.  Since the choice of which $f_i$ to
apply is determined by inequalities, we will look at the product action.
Define $f^{1+\sigma}$ to act on
$I \times \C_\sigma$ by $f^{1+\sigma}: (x, z) \mapsto (f(x), f_i^\sigma(z))$
where $f_i$ is a linear piece that $f$ applies to $x$.  In the
ambiguous case where $x$ is one
of the critical points, this definition is discontinuous, so we will look
at both images, which are the limit from the left and the limit from the 
right.

Define the \emph{boundedness set} $B(f^{1+\sigma})$ to be the set of
$(x,z)$ such that the its orbit stays bounded.  If $f$ is postcritically
finite,  then the critical points in particular have bounded orbits, 
so in addition the full orbit of the critical points (under taking
inverse images as well as forward images) are bounded.  Define
the limit set $L(f^{1+\sigma})$ to be the smallest closed set containing
all $\omega$-limit sets for $(x,z)$.
\begin{figure}[htpb]
\centering
\includegraphics[width=\textwidth]{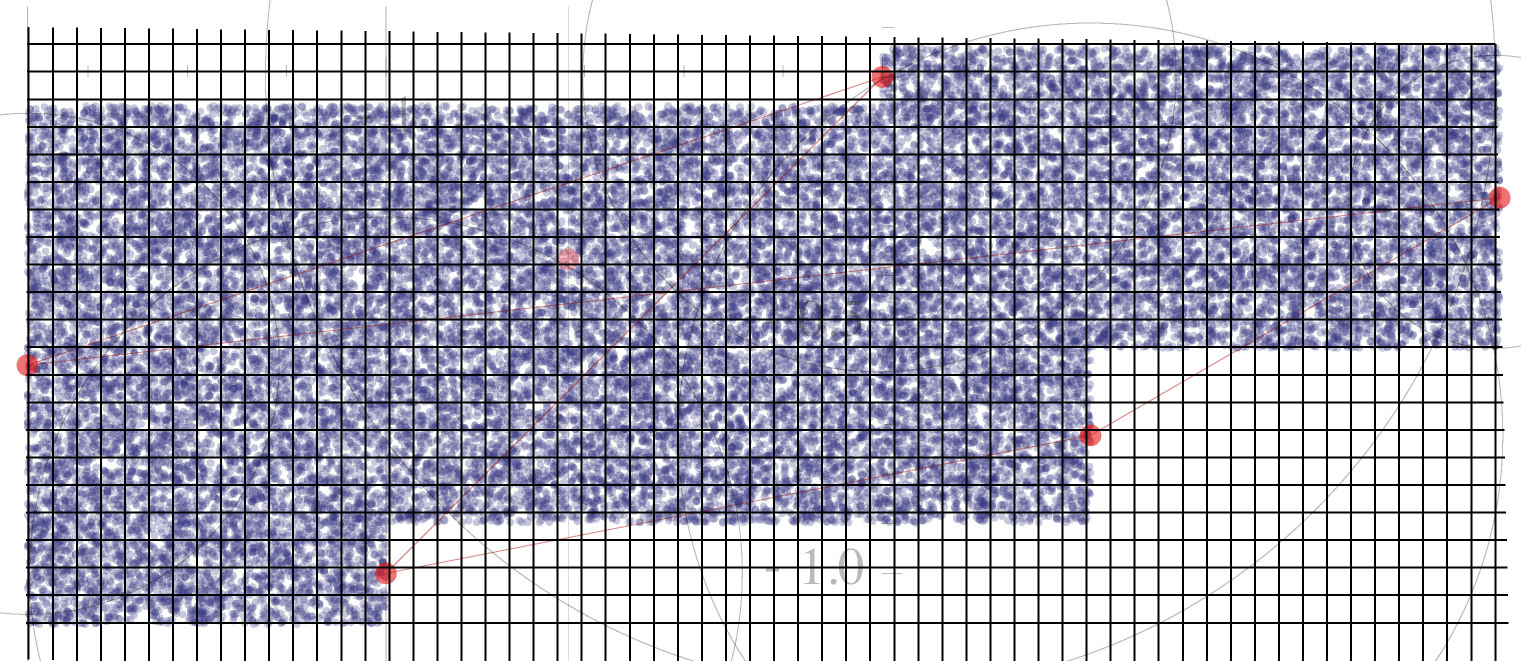}
\caption{This diagram shows $L(f^{1+\sigma})$ where
$\lambda = 1.7220838\dots$
is a Salem number of degree 4 satisfying $\lambda^4-\lambda^3-\lambda^2-\lambda+1=0$, and $\sigma(\lambda) = 1/\lambda$.
The critical point is periodic of period 5, and the 
limit set is a finite union of rectangles with a critical point on
each vertical side. The dynamics multiplies the left half (to the
left of the vertical line through the uppermost red dot)
by the diagonal matrix with entries $(-\lambda, -1/\lambda)$, then
translates until the lowermost red dot goes to the uppermost red dot.
The right half of the figure is multiplied by the diagonal matrix
$(\lambda, 1/\lambda)$, and translated to fit in the lower left.  As with
the example in figure \ref{fig:FoldedFibonacci}, it can be folded up, starting by folding the vertical sides at the red dots, to form (topologically) an $S^2$
on which the dynamics acts continuously, a pseudo-Anosov map of the $(2,2,2,2,2)$-orbifold.  This phenomenon has been explored, in greater generality, by Andr\'e de Carvalho and Toby Hall \cite{CH}, and it is also related to the concept of the dual of a hyperbolic groupoid developed by
Nekrashevych \cite{Nekrash}.
}
\label{fig:FoldedPolygon}
\end{figure}

There are three qualitatively different cases, depending on whether
$\abs {\sigma(\lambda) }$  is less than 1, equal to 1, or greater than 1.

In the first case that $\abs{\sigma(\lambda)} < 1$, every orbit remains
bounded in the $\C_\sigma$ direction, since the map  is the composition of 
a sequence
of contractions.  Everything outside a certain radius is contracted, so
the boundedness locus is compact.  Some special examples of this are illustrated
in figures \ref{fig:FoldedFibonacci}, \ref{fig:FoldedFibonacciTetrahedron} and
\ref{fig:FoldedPolygon}.

The second case, when $\abs{\sigma(\lambda)} = 1$, seems hardest to understand,
since the $f_i^\sigma$ act as isometries. Perhaps the limit set in these cases
is all of $\C$.

In the third case, when $\abs{\sigma(\lambda)} > 1$, the maps $f_i^\sigma$ are expansions, so there is a radius $R$ such that everything outside the disk $D_R(0)$ of radius $R$ centered at the origin escapes to $\infty$. Given $x \in [0,1]$,
the only $z$ such that $f^{1+\sigma}(x,z)$ has second coordinate inside $D_R(0)$ are inside a disk of radius $R/\sigma(\lambda)$ about the preimage of 0.  Starting far along in the sequence of iterates and working backwards to the beginning, we find a sequence of disks shrinking geometrically by the factor $1/\sigma(\lambda)$ that contain all bounded orbits. In the end, there's a unique point $b(x) \in \C$ such that $(x, b(x))$ remains bounded. The point $b(x)$ can also be expressed as the sum of a power series in $\sigma(\lambda)^{-1}$ with bounded coefficients that depend on the $C_i$ and  the kneading data for $x$, matching with formulas from \cite{MilnorThurstonKneading}.

\begin{proposition} \label{BoundednessContinuity}
If $f$ is postcritically finite and if $\abs{\sigma(\lambda)} > 1$, then $b(x)$
depends continuously on $x$.
\end{proposition}
\begin{proof}
When $f$ is postcritically finite, the preperiodicity of $f$ at any critical point $c_i$ is equivalent to an identity among compositions of the $f_i$. Therefore
$f^{1+\sigma}$ satisfies the same identity when applied to $(c_i, \sigma(c_i))$,
therefore its orbit is bounded so $b(c_i) = \sigma(c_i)$.

Since $b(x)$ is the sum of a geometrically convergent series, the value depends continuously on the coefficients. The coefficients change continuously except where $x$ is a precritical point; but the limits from the two sides coincide at
precritical points, since they coincide for critical points.  Therefore $b(x)$
is continuous.
\end{proof}

When $f$ is not postcritically finite, $b(x)$ might not be continuous; there could well be different limits from the left and from the right at critical points, and therefore at precritical points: the natural domain for $b(x)$ 
in general is a Cantor set obtained by cutting the interval at the countable dense set of precritical points. For example,
if $\lambda$ is transcendental, then $\sigma$ could send it to any other transcendental number, and it's obvious that generically $b(x)$ would not be continuous.

This phenomenon points to the inadequacy of considering only the algebraic properties of $\sigma$ when our real goal is to control the geometry. Here is
a formulation of the appropriate condition, in a slightly bigger context:

\begin{defn}
Suppose  $f \in F(d, \epsilon)$, 
$V$ is a vector space  and $g: [0,1] \times V \to [0,1] \times V$ 
that has the form 
\[g: (x, v) \mapsto (f(x),  \epsilon (-1)^{i} A(v) + C_i)
\]
when $c_{i} < x < c_{i+1}$,
where $A$ is an expanding map and the $C_i$ is a constant vector.
Then $g$ is a \emph{friendly extension of $f$}
if there is a continuous map $x \mapsto b(x)$ such that the orbit of $g(x, b(x))$
is bounded.
\end{defn}

\begin{figure}[htpb]
\centering
\includegraphics[width=4.5in]{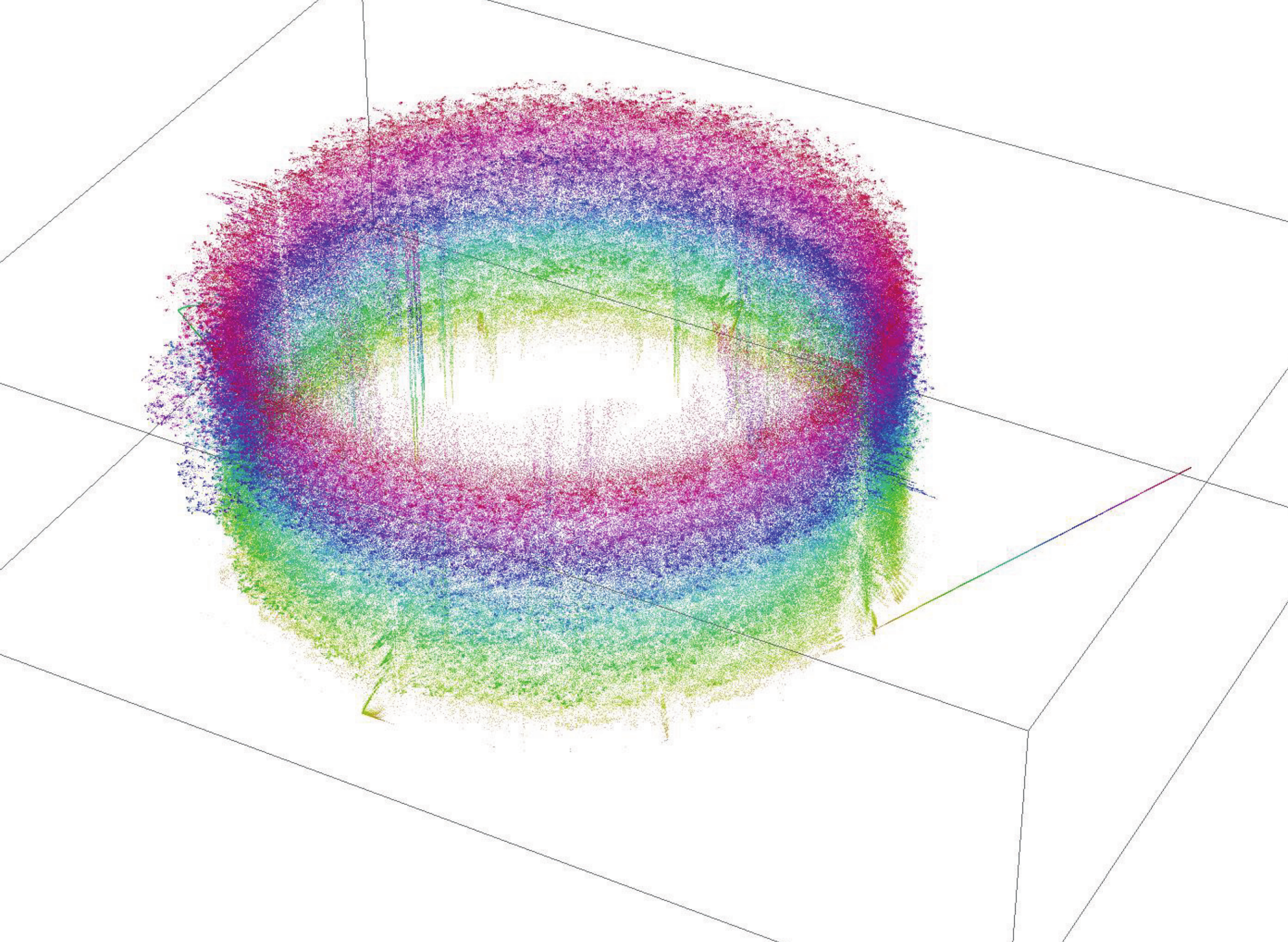}
\caption{
This picture shows the parameter values for friendly extensions of tent maps
(the space $F(1,2)$). The vertical direction is the $\lambda$-axis. The horizontal plane is $\C$, and the points shown are roots of minimal polynomials for postcritically finite examples. The points are color-coded by value of $\lambda$ (height). With a finer resolution and expanded vertical scale, the
friendly extensions would appear as a network of very frizzy hairs (usually of Hausdorff dimension $> 1$), sometimes
joining and splitting, but always transverse to the horizontal planes.
The diagonal line on the right is $\lambda$ itself.
}\label{fig:KneadingFriends}
\end{figure}

Since postcritically finite systems are dense in $F(d, \epsilon)$, there are 
many postcritically finite examples, and many have Galois conjugates of $\lambda$
outside the unit circle. Since boundedness depends
continuously on the kneading data, the coefficient $\lambda$ and the constant
terms $C_i$, we can pass to limits of postcritically finite systems;
if we take a convergent sequence of postcritically finite systems having a convergent sequence of friendly extensions, then the limit
is also a friendly extension. 

In the case $d = 2$, these correspond to the
kneading roots which were collected in figure \ref{fig:KneadingRoots}. Note that even for a $\lambda$ for which $x \mapsto \lambda \abs x - 1$ is postcritically finite, it is a limit of $\lambda_i$ of much higher degree for which  $x \mapsto \lambda_i \abs x - 1$ is also postcritically finite; in fact, it is a limit of cases where the critical point is periodic. The minimal polynomials for these nearby polynomials can be of much higher degree, so there can be many more friendly extensions than just the ones associated with the roots of the minimal polynomial for $\lambda$.
Figure
\ref{fig:KneadingFriends} is a 3-dimensional figure of friendly extensions, where the vertical axis is the $\lambda$ direction, and the horizontal direction
is $\C$; the plotted points outside the unit cylinder are expansion factors for 
friendly extensions.

Figure \ref{fig:ThinKneadingSlice} is a very thin slice of the set depicted in 
figure \ref{fig:KneadingFriends}, halfway up and of thickness $10^{-9}$. A movie
made from frames of this thickness, at 30 frames per second, would last a year.
This slab is thin enough to freeze the motion of 7 isolated friendly parameters, which you can see around the periphery. One of the 7, at position 1.5, is the original controlling expansion constant $\lambda$. Closer to the unit circle, the friendly parameters are packed closer together, and they move so quickly that they blur together into a big cloud.

\begin{figure}[htpb]
\centering
\includegraphics[width=4in]{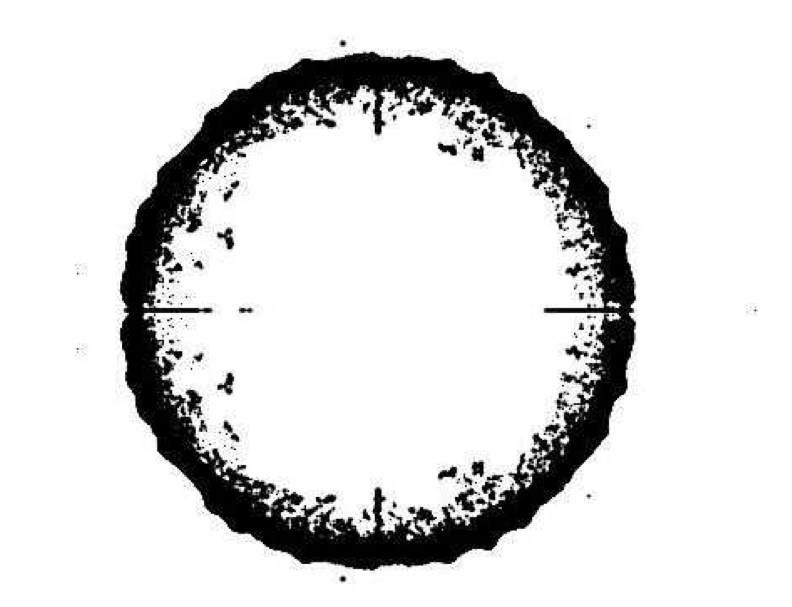}
%\sarah{KneadingRoots-fine} This file was missing
\caption{
This image shows roots of 20,000 postcritically finite polynomials of degree about 80 in a thin slab, $1.500000059 \le \lambda < 1.500000060.$ There are 7
isolated spots visible outside the unit circle, with the smallest at 1.5. This
 slice of expansion factors of length $10^{-9}$ is thin enough to confine one friendly extension to stay within  each spot as $\lambda$ changes. Roots closer to the unit circle move quite fast with $\lambda$, so that even 
 in this short interval the closer roots wander over large areas that overlap and sometimes perhaps collide and split.  Roots in the closed unit disk do not depend continuously on $\lambda$, but they are confined to (and dense in) closed sets that include the unit circle and increases monotonically with $\lambda$, converging at $\lambda=2$ to the inside portion of figure \ref{fig:KneadingRoots}.
}
\label{fig:ThinKneadingSlice}
\end{figure}

\begin{remark}
The ``set of friends'' changes continuously, regarded as an atomic measure on the parameter space for candidate systems, in the weak topology.
\end{remark}
%\sarah{commented out this section}
%\bigskip
%\dots {\bf TO BE FILLED IN}: \dots
%Statement and proof that the set of friends changes continuously, regarded as an
%atomic measure on the parameter space for candidate systems, in the weak topology.  Generalization of: roots of power series with bounded coefficients are isolated and depend continuously on coefficients.

\section{Train Tracks}\label{sec:TrainTracks}

We have defined a train track map of a graph to be a map such that each edge is mapped by a local embedding under all iterates; 
now we'll define a train track structure:

\begin{defn}
A train track structure $\tau$ for a graph $\Gamma$ is a collection of 2-element subsets of the link of each vertex, called the set of legal turns.
%{\bf pair is ordered or not ???}
\end{defn}

\begin{figure}[htbp] 
\centering\includegraphics[width=5in]{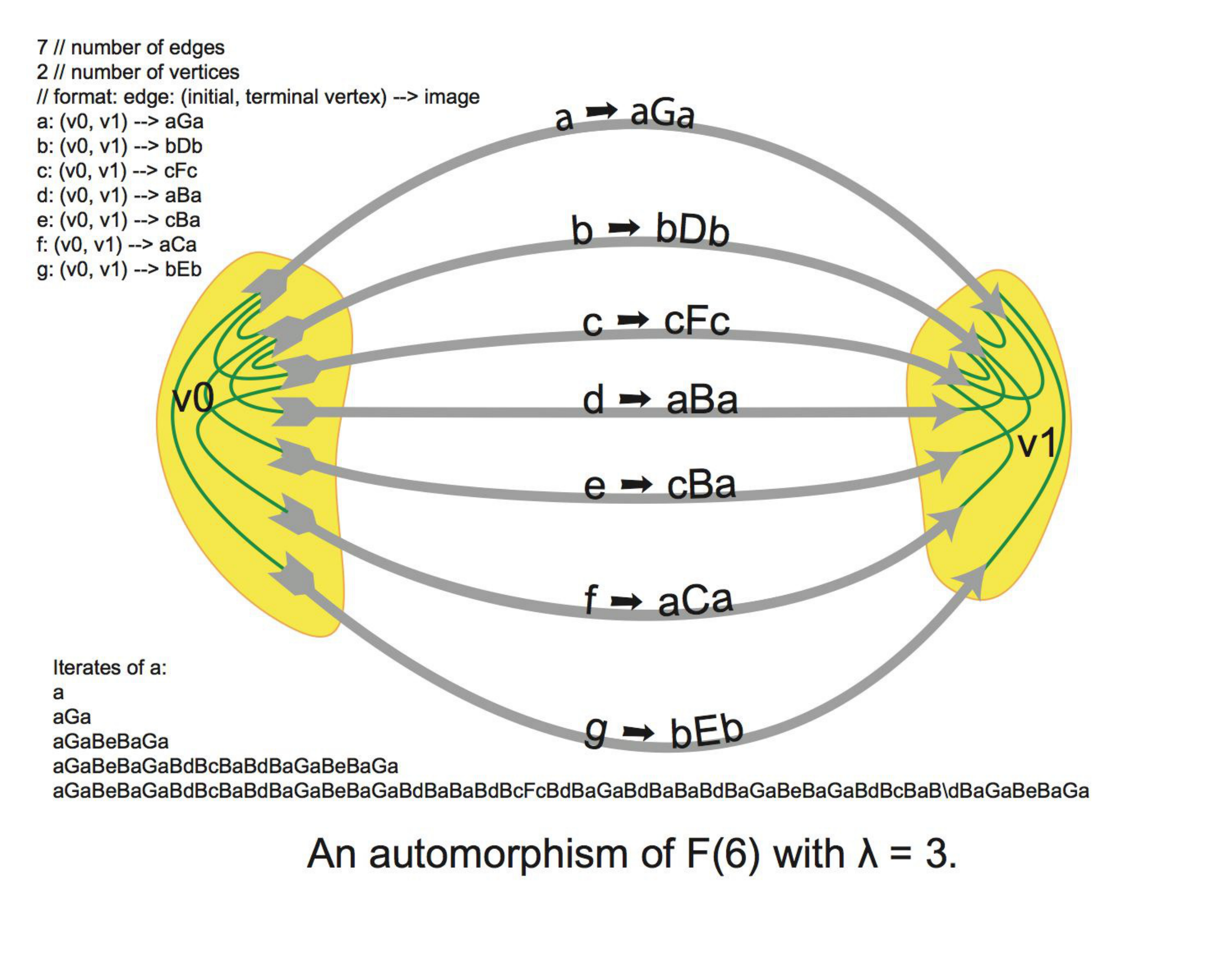}
\caption{This is a train track preserved by an automorphism of the free group of rank 6 with expansion constant 3. The edges are labeled with letters, where a capital letter indicates the reverse direction; one set of generators is
$\{bA,cA,dA,eA,fA,gA\}$.
The two vertices $v_0$ and $v_1$ are `exploded' to show legal turns, which are marked in green. This example will be used centrally for a  template to constructing train track maps
of expansion factor rate all possible Perron numbers.}
\label{fig:AutomorphismX3}\
\end{figure}
The mental image is that of a railroad switch, or more generally a switchyard, where for each incoming direction there is a set of possible outgoing directions where trains can be diverted without reversing course.  A path on $\Gamma$ is \emph{legal} if it is a local embedding, and at each vertex it takes a legal turn.

In describing edge paths, we must first choose an orientation for each edge.
We use the convention that a lower case letter denotes the forward direction on the edge, and the corresponding
capital letter to denote the backward direction on the edge.  This convention also applies to generators of groups (which we can think of as edges in a cell  complex having the given group as its fundamental group).

A map $f: \Gamma \rightarrow \Gamma$ \emph{preserves} $\tau$ if $f$ maps every $\tau$-legal path to a $\tau$-legal path.   It follows that $f$ is a train track map: since the forward images of edges must always be legal paths, in particular they are mapped by local embeddings.

For any train track map $f$, there is always at least one train track structure that $f$ preserves.  The \emph{minimal invariant train track structure} allows turns only if they are ever taken by the forward image of some edge of the graph. The \emph{maximal invariant train track structure} allows all turns that are never folded by iterates of $f$, that is, it allows any turn that is always mapped to be locally embedded.

There is an interesting special case of train track structure for
a bouquet of circles:
the positive train track structure, where the legal paths correspond to paths that are positive words in the generators of the free group. These paths follow a consistent orientation along the circles.   More generally, a train track structure is \emph{orientable} if there is no legal path that crosses an edge twice in opposite directions; orientability holds if an orientation of the edges can be chosen so that all legal paths maintain a consistent direction (but the converse need not hold).

% in that case, an orientation can be chosen so that all legal paths maintain a consistent orientation (but the converse need not hold).

An oriented train track on a graph $\Gamma$ defines a convex cone  $C$ in $H^1(\Gamma, \Z)$. For an automorphism $\phi$ that preserves such a structure, $H^1(\phi, \R)$ takes $C$ to $C$ or to $-C$. Therefore, $\pm \phi$ has an
eigenvector inside $C$. Just as in the Perron-Frobenius theorem, an eigenvector
strictly inside $C$ has the largest eigenvalue, and in any case, there is an eigenvector in the cone (possibly on its boundary) whose eigenvalue is largest.
This eigenvalue is the same as the expansion constant for $\phi$, since the
incidence matrix for $\phi$ is the matrix for $\pm \phi$ acting on 
(simplicial) 1-chains of $\Gamma$. As an eigenvector for an element of $\GL(n,\Z)$, such an eigenvalue is always an algebraic unit.

As illustrated by the example of figure \ref{fig:AutomorphismX3}, the expansion constant for a train track map need not be a unit: the figure describes a map $\phi_3$ with expansion constant 3.  

To check that it is an automorphism, first write down the images of the generators
$\{bA,cA,dA,eA,fA,gA\}$ and collapse the $a$ edge to a point (thus striking out all $a$'s):

\begin{align}
\phi_3: b \mapsto & bDbg \cr
c \mapsto & cFcg \cr
d \mapsto & Bg\cr
e \mapsto & cBg\cr
f \mapsto & Cg\cr
g \mapsto & bEbg
\end{align}

Every surjective selfmap of a free group
is an automorphism, so it's enough to check that all 6 generators are in the image.  We get $c$ from $eD \mapsto c$, and given $c$ we get $g$ from $f \mapsto Cg$. Given $c$ and $g$, we get $b$ from the image of $E$ and $f$ from the image of $C$. Given these four generators we get the remaining two, $d$ from the image of $B$ and $e$ from the image of $G$.  Therefore, $\phi_3$ is a self-homotopy-equivalence of the graph, so it gives a train track map.

To see that the map preserves the train track structure, first note that every edge maps to a legal path.  Now note that the three edges $a, b, c$ each map to a path starting and ending in the same way, with $a, b$ or $c$ respectively.
All words only involving these three edges are legal for this train track.
The other four edges map to words beginning and ending with $a, b$ or $c$,
so it is easy to check that legal turns are mapped to legal turns.

This implies that there is no orientable train track structure invariant by
the automorphism up to homotopy.  For this particular example, the structure has
an anti-orientation, such that every legal path reverses orientation at every vertex.

\medskip
  For $m \ge 0$, a map $\phi_{3+2m}$ with expansion factor $3+2m$
can be defined, preserving the same train track:
\begin{align}
\phi_{3 + 2m}: a &\mapsto aG(aB)^m a \cr
b &\mapsto bD(bC)^m b\cr
c &\mapsto cF(cA)^m c\cr
d &\mapsto aB(aB)^m a\cr
e &\mapsto cB(aB)^m a\cr
f &\mapsto aC(aB)^m a\cr
g &\mapsto bE(bA)^m b  \label{PhiFormula}
\end{align}

A sequence of steps identical to that for $\phi_3$ shows that $\phi_{3+2m}$ is a
train track map.     For completeness, we can define $\phi_1$ as the identity map of this train track; $\phi_1$ is also a train track map.  There are  similar constructions for even expansion factors, but we do not need them.

\section{Splitting Hairs} \label{sec:SplittingHairs}

Let $f: \Gamma_0 \rightarrow \Gamma_0$ be any purely expanding selfmap of a tree.
Every tree is bipartite, so we can partition the vertices into to sets $V_0$ and
$V_1$ such that there are no edges of $\Gamma_0$ with both endpoints in one of the sets. 

The map $f$ may not respect this partition; if not, let $\Gamma$ be the barycentric subdivision of $\Gamma_0$, with one new vertex $v(e)$ in the middle of each edge $e$ of $\Gamma_0$. For each edge $e$ of the original $\Gamma_0$, choose one of the edges $e'$ in its image, and 
adjust $f$ so that $v(e) \mapsto v(e')$. The new map still has a non-negative incidence matrix. 
Any positive linear function that is an eigenvector for its transpose pulls back to a linear function that is an eigenvector for the transpose of the incidence matrix of the original $f$, so the eigenvalues are the same.
%Any positive linear function that is an eigenvector for its transpose pulls back to a linear function that is an eigenvector for the original $f$, so its eigenvalue is the same.
 Therefore, there are positions for the new vertices $v(e)$ so all vertices of $\Gamma$ are mapped (by the original $f$) to vertices of $\Gamma$.   The graph $\Gamma$ has a bipartite partition,
where $V_0$ consists of original vertices and $V_1$ consists of new vertices,
where each partition element is invariant by $f$.

Given any graph $\Gamma$ and an integer $n$, define a new graph $\Split_n(\Gamma)$ to be obtained by replacing each edge of $\Gamma$ with
$n$ edges having its same pair of endpoints.  For each edge $e$, label
the new edges with subscripts $e_a, e_b, \dots$.   

When $\Gamma$ has a bipartite structure $\{V_0, V_1\}$, define a train track structure on $\Split_7(\Gamma)$ using the prototype of figure \ref{fig:AutomorphismX3}: that is, a path is legal 
if and only if the sequence of subscripts defines a legal path in the prototype.

When $f$ is a map
of $\Gamma$ to itself that preserves the partition elements and maps each edge to an edge-path of length at least one, define 
\[\Split_7(f): \Split_7(\Gamma) \rightarrow \Split_7(\Gamma)
\]
using the prototypes $\phi_i$ defined in the preceding section \ref{sec:TrainTracks}.  That is, for an edge $e_x$ of $\Split_7(\Gamma)$,
if $f(e)$ has combinatorial length $k$, lift $f$ to the map $\Split_7(f)$ by applying
the sequence of subscripts $\phi_k(x)$ to the sequence of edges $f(e)$.

Note that the $\phi_{2m+1}$ themselves are defined by this process, starting with the self-maps of the unit interval that fold it over itself an odd number of times.

\begin{proposition}
For any bipartite structure $\{V_0, V_1\}$ on a graph $\Gamma$ and any map $f: \Gamma \rightarrow \Gamma$
that preserves the partition elements and maps each edge to an edge-path of length at least one, 
$\Split_7(f)$ is a train track map.
\end{proposition}
\begin{proof}
Notice that in the prototype train track, every legal turn has at least one of its ends among the edges $a,b,c$.  Any turns among these three edges are legal,
and every $\phi_{2m+1}$ preserves their beginnings and end. Every  $\phi_{2m+1}$
with $m > 1$ maps beginnings and ends in the same way, so a turn between edges whose combinatorial image length is more than 1 maps to a legal turn.  Similarly, a legal turn between edges one or both of which have combinatorial image length 1 maps to a legal turn.
\end{proof}

\begin{proposition} \label{LiftingTreeMaps}
If $\Gamma$ is a  tree with bipartite structure $\{V_0, V_1\}$,
and if $f: \Gamma \rightarrow \Gamma$ is self-map
such that 
\begin{itemize}
\item{f preserves $V_0$ and $V_1$, and}
\item{f is a local embedding on each edge, and}
\item{the map $f_1: e \mapsto f(e)_1$, that is, $e$ goes to the first element of its image edgepath, is a permutation}
\item{the map $f_2: e \mapsto f(e)_2$ if $f(e)$ has length more than 1, and $e \mapsto f(e)_1$ otherwise, is also a permutation}
\end{itemize}
then $\Split_7(\Gamma)$ is a homotopy
equivalence.
\end{proposition}
\begin{proof}
The set of edges with subscript $a$ forms a spanning tree for $\Split_7(\Gamma)$.
If we collapse the spanning tree to a point, we obtain a bouquet of circles, so the remaining edges give a free set of generators for the fundamental group.

Pick a basepoint $*$ on the graph $\Split_7(\Gamma)$.  As an edge path,
the generator corresponding to $e_x$ is obtained by prefixing $e_x$ by the
subscript $a$ path from $*$ to its first vertex, and appending the $A$ path from its end vertex back to $*$, then striking out all $a$'s and $A$'s.

As before, we just need to show that all generators are in the image of $\Split_7(f)$.  Consider the 6 generators lying over any particular edge $x$ of $\Gamma$.   Let $x'$ be the first edge in the edgepath $f(x)$.  We will follow the same outline that showed $\phi_3$ is a homotopy equivalence.  The image of $x_e X_D$ under $\Split_7(f)$ is $x'_c X'_A$ which becomes $x'_c$ after the collapse. %The image of $f(x_e X_D)$ is $x'_c$; 
Applying this to all edges $x$ of $\Gamma$, we obtain
all generators $x_c$ in the image.  Modulo the $x_c$ generators, from $\Split_7(f)(x_f)$ we get all generators of the form $f_2(x)_g$. Since $f_2$ is a permutation, this gives all subscript $g$ generators.  From $\Split_7(f)(X_D)$ modulo the $g$ generators, we get the $f_1(x)_b$ generators.  Continue in the same sequence that was used to show $\phi_3$ is surjective: modulo previous generators, each edge has a payload generator in the first or second slot of its image, so we get all generators.
\end{proof}

Given any Perron number $\lambda$, we can now apply proposition \ref{LiftingTreeMaps} to the asterisk map $f_\lambda$ constructed by theorem \ref{AsteriskMaps}, where we take $V_0$ to consist of the center vertex and $V_1$ to consist of all the tips. We obtain a train track map $\Split_7(f_\lambda)$.
The maps were constructed so that every edge occurs as the first segment of
some image edge-path. The same edge occurs as the second segment, if the path has length more than 1, so the maps $f_{\lambda,1}$ and $f_{\lambda,2}$ are both permutations.
Therefore, $\Split_7(f_\lambda)$ is  train track homotopy equivalence of $\Split_7(\Gamma)$ which uniformly expands with expansion factor $\lambda$.

The maps $\phi_{2m+1}$ are mixing when $m > 1$, from which it follows that $\Split_7(f_\lambda)$ is mixing.

If $\lambda$ is a weak Perron number, let $n$ be the least integer such that $\lambda^n$ is a Perron number.   Make an asterisk from $n$ copies of an asterisk for $\lambda^n$, and map it to itself by permuting the factors except for the final map, which is a copy of $f_{\lambda^n}$.   From this asterisk map
$f_\lambda$, we obtain a train track
automorphism $\Split_7(f_\lambda)$ with expansion constant $\lambda$.
This completes the proof of theorem \ref{Main2}.

\section{Dynamic Extensions} \label{sec:DynamicalExtensions}

If $f: X \rightarrow X$ is a continuous map, then an \emph{extension} of $f$ is a space $Y$ with a continuous surjective map $F: Y \rightarrow Y$ and a semiconjugacy $p: Y \rightarrow X$ of $F$ to $f$, that is, it satisfies $p\circ F = f \circ p$.

The hair-splitting construction of section \ref{sec:SplittingHairs} is a special case, which we'll call a \emph{graph extension}, where $X$ and $Y$ are graphs and edges of $Y$ are mapped to non-trivial edge-paths of $X$.   If $f$ is a $\lambda$-uniform expander, then so is $F$. Given two graph extensions $F_1$ and $F_2$, they have a fiber product $F_3$, consisting of all pairs of points in $\Gamma_1$ and $\Gamma_2$ that map to the same point in $\Gamma$.  We can pick a connected component of the fiber product to obtain a connected graph that is a common extension of both; thus the set of connected graph extensions is a partially ordered set where any two elements have an upper bound (but not necessarily a least upper bound).

It seems natural to ask which uniformly expanding selfmaps of graphs have graph extensions that are train-track self-homotopy-equivalences.  It would also be interesting to strengthen the condition, to require that the fundamental group of the extension graph maps surjectively to the fundamental group of the base; in that case, a necessary condition is that $f$ itself be a homotopy equivalence, otherwise neither it nor its extension would be surjective on $\pi_1$.  If necessary, we could also weaken the condition by allowing a subdvision of $\Gamma$ before taking the extension.

Although proposition \ref{LiftingTreeMaps} required that the first and second elements in the edgepaths for $f$ are partitions, this condition does not seem essential.  One trick would be to modify the formula of \ref{PhiFormula} to lift $f$, we could 
use lifts in patterns such as $b \mapsto b (Cb)^{m1} D(bC)^{m2} g$  and
$e \mapsto (aB)^{m1}c(Ba)^{m2}$ to adjust the payload edges to be somewhere else along the edgepath, taking care that the payload edges map surjectively, and that they are chosen so that there is an order that will unlock them inductively. 

It seems plausible that by a combination of duplicating edges, subdividing edges, and lifting with adjustable payloads, any expanding self-homotopy-equivalence of a graph would have an extension that is a train track map of a graph whose fundamental group maps surjectively to the base. However, it  is beyond the scope of the current paper to pursue this.

It would also seem interesting to understand a theory of minimal uniformly expanding maps, ones that cannot be expressed as non-trivial extensions. Perhaps extensions that merely identify vertices should be factored out. This should be related to a theory of finitely generated abelian semigroups that exclude 0 and are invariant by a transformation and which, tensored with $\R$, have $\lambda$ as dominant eigenvalue. {\em They can be thought of as a positive $\lambda$-modules}.

\section{bipositive matrices} \label{sec:BipositiveMatrices}

\begin{defn}
A pair of real numbers $(a,b)$ is a \emph{conjugate pinching pair} 
if $a$ and $b^{-1}$ are
algebraic units whose Galois conjugates are contained in the open annulus of inner radius $b^{-1}$ and outer radius $a$.   They are a \emph{weak conjugate pinching pair} if their conjugates are contained in the closed annulus.
\end{defn}

Here is an elementary fact:
\begin{proposition}
\label{GL expansions}
A real number is the expansion constant for an element $A$ of $\GL(n,\Z)$ for some $n$ if and only if it is a weak Perron number.
%conjugate bounder.

A pair $(a,b)$ of algebraic units occurs as the pair of expansion constants 
for $A$ and $A^{-1}$
if and only if the pair is a weak conjugate pinching pair.
\end{proposition}
\begin{proof}
Sufficiency is easy. Given a weak Perron number $a$, the companion matrix for the
minimal polynomial for $a$ has expansion constant $a$.  Given a conjugate pinching pair $(a,b)$, take the sum of the companion matrix for a minimal polynomial for $a$ with the inverse of a companion matrix for the minimal polynomial for $b$.

Necessity is pretty obvious:  $a$ must
be the maximum absolute value, and $b$ the minimum, among all
roots of $\chi_A$.
\end{proof}
\begin{defn}
An invertible matrix $A$ is \emph{bipositive} with respect to
a basis $B$ if $B$ admits a partition into two parts $P$ and $Q$ such that
$A$ maps the orthant spanned by $P \cup Q$ to itself, and $A^{-1}$
maps the orthant spanned by $P \cup -Q$ to itself.
\end{defn}

If either of the two positive matrices associated with
a bipositive matrix is mixing,  the other is as well, so in this case
we also say
the bipositive matrix is mixing.
Similarly, we will call a bipositive matrix \emph{mixing} if
all sufficiently high powers have all non-zero entries. More
geometrically, given a positive matrix, we can make a graph whose
vertices are the basis elements and whose edges correspond to nonzero
entries. The matrix is primitive if there are edge-paths of every
sufficiently long length from each vertex to each others.  The set
of lengths of edge-paths from one vertex to another is a semigroup,
and it is easy to see that its  complement is finite
if and only if the elements have a
common divisor larger than one.
It follows that an irreducible positive or bipositive
matrix $A$ is mixing
if and only if every positive power of $A$ is ergodic.

\begin{theorem} \label{bipositive matrices}
A pair $(a,b)$ is the pair of positive eigenvalues for an invertible bipositive integral matrix of some dimension
if and only if it is a conjugate pinching
pair, or a weak conjugate pinching pair such that
all conjugates of $a$ and $b^{-1}$ of modulus $a$ or $b^{-1}$
are $a$ or $b^{-1}$ times roots of unity.
\end{theorem}

\begin{remark}
The clause about roots of unity addresses the issue that although all Galois conjugates of a weak Perron number $a$ that have \emph{maximum} modulus are at angles that are roots of unity, the Galois conjugates of \emph{minimum} modulus need not be.

For instance, if $a$ is the plastic number $1.32472\dots$ which is a root of $x^3-x-1$, its other two conjugates $-0.662359\dots \pm i 0.56228 \dots $ have modulus $1/\sqrt a$, so $(a, \sqrt a)$ is a weak conjugate pinching pair. However, the two complex Galois conjugates of $a$ are at angles that are irrational multiples of $\pi$.  Therefore $(a, \sqrt{a})$ is not the pair of expansion constants for any bipositive matrix.
\end{remark}

\begin{proof}
We will first establish the theorem in the case $(a,b)$ is a [strict]
conjugate pinching pair. 

First, let $A \in \GL(n,\Z)$
have eigenvectors $S$ and $T$ of
eigenvalues $a$ and $b^{-1}$. By restricting to the smallest rationally defined subspace of $\R^n$ containing $S$ and $T$, we may assume that all other
characteristic roots of $A$ are in the interior of
the annulus $b^{-1}<\abs x < a$. Let $H$ be the real subspace
spanned by $S$ and $T$, and $p: \R^n \rightarrow H$ be the projection
that commutes with $A$.

Let $L$ be the line through the origin
of slope $ a*b$ in the $ST$-plane.  The linear transformation $A$ multiplies
slopes by $1/a*b$, so $A(L)$ has slope $1$. Let $L'$ be the reflection
of L in the $S$-axis; thus $A(L')$ is the reflection of $A(L)$.

Choose a set $P_0$  of lattice points such that $p(P_0)$ lies
in the interior of the first quadrant of the $ST$-plane, $P_0$
generates the lattice $\Z^n$ as a group, and the convex cone
generated by $P_0$ contains the angle between $L$ and $A(L)$, except $0$,
in its interior.  Similarly,
choose $Q_0$ to be
a set of lattice vectors that generate $\Z^n$ as a group,  that are mapped
by $p$
to the interior of the fourth quadrant, generates $\Z^n$ as
a group, and contains the angle between $L'$ and $A(L')$.

Let $\sigma_P$ be the semigroup generated by $ P_0 $
and let $\sigma_Q$ be the semigroup generated by $Q_0 $.

We claim that for sufficiently large $k > 0$ and for any two elements
$p_1, p_2 \in P_0$ there is a $j > 0$ such that 
$$ A^j(p_2) - A^{-k} (p_1) \in \sigma_Q, $$
and similarly, for any $q_1, q_2 \in Q_0$ there is a $j > 0$ such that
$$ A^j(q_2) - A^{-k}(q_1) \in \sigma_P.$$
To see this, we will make use of a basic fact about the
semigroups of lattice
elements:

\begin{lemma} \label{semigroup most of cone}
Let $U$ be a finite set of elements of the
lattice $\Z^n \subset \R^n$,
and let $C_U$ be the convex cone generated by $U$.
There is some constant $R$ such that any element $g$ of
the group $G_U$ generated by $U$
whose radius $R$ ball $B_R(g)$ is contained in $C_U$ is also
contained in the 
semigroup $S_U$ generated by $U$.
\end{lemma}
\begin{proof}
The convex cone $C_U$ is the set of all non-negative real linear combinations
of elements of $U$. If we express an element of $G_U \cap C_U$ as a non-negative real linear combination of elements of $U$ and round the coefficients to the nearest integer,
we see that $g$ is within a bounded distance $R_1$ of some element of $S_U$.

Let $F$ be the set $G_U \cap B_{R_1}(0)$. Express each element of $f \in F$ as an integer linear combination $\sum_{u \in U} k_{u,f} u$. Let $R_2$ be the maximum,
among  these linear combinations, of the norm of the set of negative coefficients, and set $R = R_1 + R_2$.

Now for any element $g \in G_U$ such that $B_R(g) \subset C_U$, there are elements $s \in  S_U$ and $f \in F$ such that $g = s + f$.

\end{proof}

Now for $k$ large, the image in projective space
of $A^{-k}(p_1)$ is close to the image of the eigenspace $T$.
The images $A^j(p_2)$ march in the direction of the $S$ axis, projectively
converging to the $S$ eigenspace, with the slope in the $ST$-plane
of the line from $A^{-k}(p_1)$ to $A^j(p_2)$ decreasing by a factor
of $a*b$ at each iterate.  If $k$ is suitably large, at least one of these iterates is the center of a large ball captured in the interior of the cone $A^{-k}(p_1)+C_Q$ and so, by the lemma, in the semigroup $\sigma_Q$.
%If $k$ is suitably large,
%at least one of these iterates is captured
%in the interior of the cone $A^{-k}(p_1) + C_Q$ and so, by the lemma,
%in the semigroup $\sigma_Q$.

\medskip
Now we can describe an irreducible bipositive matrix. Take the free 
abelian group $FA$
generated by
$$\bigcup_{g \in {P_0\cup Q_0 }}\left \{ A^{-k}(g), \dots, A^{j(g) - 1}(g) \right\}.
$$
We will
choose a bipositive map $\tilde A$
of $FA$ to itself that commutes with evaluation in $\Z^n$.
The
generators $A^h(g)$ pass off from one to the next
 until $ A^{j(g) - 1}(g)$. Choose a cyclic
permutation of $p$ of $P_0$ and a cyclic permutation $q$ of $Q_0$.
For each $g \in P_0$ choose an expression $A^{j(g)}(g) = A^{-k}(p(g)) + sq$
where $sq$ is an element of $\sigma_Q$, and similarly
for $g \in Q_0$ express $A^{j(g)}(g) = A^{-k} (q(g)) + sp$, 
with $sp \in \sigma_P$.  Use this to give the final links in the chain,
to define $\tilde A$.

The matrix for the linear transformation can be expressed as an upper triangular
matrix followed by a permutation, hence it is invertible. In block form,
the $P$ generators are each expressed as another $P$ generator plus an
element of the $\sigma_Q$ semigroup, and \emph{vice versa}. The inverse
has the same form, but with semigroup elements subtracted; reversing the 
sign of the $Q$ generators turns the inverse into a positive matrix.

The matrix for $FA$ is irreducible because of the cyclic permutations:
the images of each generator eventually involves each other generator,
and (except in the trivial case $a=b=1$) the powers of the matrix are
eventually strictly positive.  The positive eigenvalue for $FA$ is $a$,
since projection to the $S$ axis gives a linear function, positive on
the positive orthant for $FA$, that is a dual eigenvector of eigenvalue $a$.
Similarly, the positive eigenvalue for $FA^{-1}$ is $b$.

\bigskip

If $(a,b)$ is a weak conjugate pinching pair such that all conjugates of $a$ or $b^{-1}$ on the circle of maximal or minimal modulus have arguments that are
rational multiples of $2 \pi$,
let $h$ be a common multiple of the denominators. Then $(a^h, b^h)$ is a strict
conjugate pinching pair; let $A_0$ be a matrix realizing this pair of pinching
constants. Take the direct sum of $h$ copies of the underlying vector space,
permute them cyclically, with return map $A_0$. This is a bipositive
matrix realizing $(a,b)$.

\bigskip

For the converse: consider any bipositive matrix $A$ with positive
eigenvalue $a$ and positive eigenvalue $b$ for $A^{-1}$ where one or both are not the unique characteristic roots of maximum modulus {\em say there are $k$ roots of modulus $a$}. Then there is a $k$-dimensional subspace with a metric where
$A$ acts as a similarity, expanding by a factor of $a$. Hence, the projective image of this subspace is mapped isometrically. Its intersection with the image of the positive orthant is a polyhedron mapped isometrically to itself; hence, its vertices are permuted. It follows that all characteristic roots of modulus $a$ have the form of $a$ times an $m$th root of unity.
\end{proof}

A simple example of the construction for a weak conjugate pinching pair is the Fibonacci transformation
$x \mapsto y$ and $y \mapsto x+y$, with eigenvalues
 $\phi=(1+\sqrt 5)/2$, the golden ratio and $-1/\phi$. The square of this transformation is $x \mapsto x+y$ and $y \mapsto x+2 y$, which is bipositive: its inverse maps the second quadrant to itself. This translates into a
 bipositive map for $(\phi, 1/\phi)$ in dimension 4, that expresses a bipositive
 linear recurrence for four successive terms of the Fibonacci sequence,
 $$x_1 \mapsto x_2, \; x_2 \mapsto x_3, \; x_3 \mapsto x_2 + x_3, \; x_3 \mapsto x_1 + 2 x_2. $$
 As another example,  $a = (1+\sqrt 2)^{1/3}$ and $b = (\sqrt 2 - 1)^{1/3}$) can be realized by a bipositive transformation in dimension 12.

\begin{question} (suggested by Martin Kassabov)
Suppose $A \in \GL(n,\Z)$ has dominant eigenvalue $a$ and $A^{-1}$ has dominant eigenvalue $b$ where $(a,b)$ is a [strict] conjugate pinching pair. Is there a basis for which some
power of $A$ is bipositive?
\end{question}

\begin{remark}
Although elementary matrices generate $\mathrm{SL}(n, \Z)$, and together with permutations generate $\GL(n,\Z)$, the semigroup is a very different matter:
for $n \ge 3$, the elementary positive semigroup 
is not even finitely generated.

To see the gap between  the elementary positive semigroup and the full
positive semigroup of $\GL(n,\Z)$, let's focus on the case $n=3$
(The  embedding in $\GL ({3+n} ,\Z)$ that fixes all but the first 3 basis elements gives examples, albeit atypical, for arbitrary dimension $\ge 3$).  

Let's look at the action of these semigroups
on the basis triangle $B$
in ${\mathbb R}{\mathbb P}^2$.
The image of the triangle by a word in the generators gives a sequence of subtriangles of this triangle, where at each step you bisect one of the sides and throw one half away.   In particular, each possible proper image
is contained in one of six half-triangles of $B$. 

Now consider in general
positive element of $\GL (3, \Z)$. It maps the tetrahedron spanned by $0$ together with the basis elements  to a `clean' lattice tetrahedron, that intersects lattice points only at its vertices. This property (clean) characterizes the possible images.  Furthermore, any clean triangle in the positive
orthant with one vertex at the origin can be extended (in many ways)
to a clean tetrahedron:
just add any vertex in one of the two lattice planes neighboring the plane containing the triangle
%just add any vertex in one of the two neighboring lattice plane to the plane containing the triangle.   

But there are  many clean triangles; in fact, the set of  lattice points in the positive orthant that are primitive (\emph{i.e.} the 1-simplex from the origin to the point is clean) have density 
$(1- 2^{-3}) * (1-3^{-3}) * (1-5^{-3}) \dots = 1/\zeta(3) \approx .831907\dots$,
and given a primitive lattice point $p$, the density of lattice points $q$
such that the triangle $\Delta (0,p,q)$ is clean is $1/\zeta(2) = 6/\pi^2 \approx .607927 \dots$.
It follows that every 
line segment contained in $B$ can be approximated in the
Hausdorff topology by an image of $B$ under the positive semigroup of 
$\GL (3, \Z)$.
Many such line segments cross all 3 altitudes of $B$, so a
positive element of $\GL (3 ,\Z)$ that maps $B$ to a nearby triangle is not in the positive elementary semigroup.

Nonetheless, images of $B$ under the positive semigroup of $\GL (3 ,\Z)$ are quite restricted. For instance, it's easy to see that the centroid of $B$, corresponding
to the line $x = y = z$, cannot be in the interior of any image of $B$. For any
finite collection of triangles not containing the centroid in their interior, most lines through the centroid are not contained in any one of them. Therefore, no finite set of positive $\GL (3, \Z)$ images of $B$ cover all possible images.
\end{remark}

Note that the proof for theorem \ref{bipositive matrices} actually
gave something stronger. If we have a basis $B$ that is partitioned into two
parts $P$ and $Q$, then any elementary transform that replaces an element of
$P$ by its sum with an element of $Q$, or an element of $Q$ by its sum with an element of $P$ is bipositive. The \emph{elementary bipositive semigroup}
[with respect to $(P,Q)$]
is the semigroup generated by these cross-type elementary transformations, together with permutations that preserve $P$ and $Q$.

\begin{theorem}\label{elementary bipositive expansions}
A pair of real algebraic units $(a,b)$ is the pair of expansion constants for an
elementary bipositive matrix and its inverse if and only if it is a conjugate pinching pair such that all Galois conjugates of $a$ or $b^{-1}$ of maximal or 
minimal modulus have arguments that are roots of unity.
\end{theorem}
\begin{proof}From the proof of theorem \ref{bipositive matrices},
the condition on arguments is an equivalent form of the hypothesis
in the case that $(a,b)$ does not strictly pinch.
\end{proof}

\section{Tracks, Doubletracks, Zipping and a sketch of the Proof of Theorem \ref{Main3}}\label{sec:Doubletracks}

Continuous maps are often inconvenient for representing homotopy equivalences between graphs, because a self homotopy equivalence  cannot  be made into a self-homeomorphism of any graph in the homotopy class unless it has finite order up to homotopy.

Continuous maps of graphs can be inconvenient as geometric 
representatives of group automorphisms of the free group, since they are usually not invertible.  As we have seen, it is not easy to see at a glance whether a given map of graphs is a homotopy equivalence. There are algorithms to check, but they can be tedious.

As an alternative,  we can represent self-homotopy equivalences
by continuous 1-parameter families of graphs. These have the advantage of being
reversible. If we restrict to
graphs that have no vertices of valence 1, these can be 
locally described by moving attachment points of edges along paths
in the complement of the edge.

A \emph{zipping} of a train track $\tau$ to a train track $\sigma$  is a 1-parameter family of structures
$(\Gamma_t, \tau_t) | \tau \in [0,1]$ that may be thought of
as squeezing together legal paths. It's elementary to see
that for any train track map $f$, there
is a zipping  that yields the homotopy class of $f$: just progressively and 
locally zip together the identifications that will be made by the map.

A zipping can be translated into a sequence of
reversible steps, consisting of a motion of an attachment point of one edge along a legal path on its complement, 
starting in a direction in its linkgroup. When $(\Gamma, \tau)$
is zipped to $(\Gamma', \tau')$, every bi-infinite $\tau$-legal path
becomes a bi-infinite $\tau'$-legal path.
The inverse of a zipping is 
an \emph{unzipping} or \emph{splitting}.

A \emph{doubletrack} structure for $\Gamma$ is a pair $(\sigma, \tau)$
of train track structures on $\Gamma$. A graph $\Gamma$ equipped with
a pair of train track structures is a \emph{doubletrack}.

A \emph{bizipping} between
doubletracks,
is a 1-parameter family of doubletracks that is a zipping of the first train track structure and an unzipping of the second.  This yields a train track map of one structure whose inverse is a train track map for the other structure.

\begin{remark}
It seems likely that invariant
foliations could provide a good alternative to train tracks.
Bestvina and Handel introduced a concept of train tracks relative to
an invariant filtration of a graph, and showed that relative train track
maps exist for every outer automorphism of a free group (not just in the
irreducible case). Instead, one could look at  foliations of finite
depth on a manifolds of sufficiently high dimension (as a function of the
rank of the free group), with singularities having links based on
polyhedra,  and satisfying the condition (used to great effect by
Novikov) that there are no null-homotopic
closed transversals.  Such a foliation picks out a class of bi-infinite
words in the free group. Bestvina-Handel's theorem on existence of
relative 
train tracks would appear to translate to the existence of a homeomorphism 
of some open manifold homotopy-equivalent to a bouquet of circles
that preserves such foliation.  Pairs of foliations could
substitute for doubletracks.

We will not take the detour of trying to develop this point of
view here.
\end{remark}

We now sketch the proof of theorem \ref{Main3}.
\begin{proof}
We will now analyze the [main] case when there are strict inequalities:
let $(a,b)$ be a strictly pinching pair of algebraic units.

From theorems \ref{bipositive matrices}, let A be a bipositive matrix with 
positive eigenvectors $S$ for $A$ and $T$ for $A^{-1}$ having
eigenvalues $a$ and $b$. Let $p$ be the invariant
projection of the vector space for $A$
to the plane spanned by $S$ and $T$. The basis is partitioned into two
sets, $P$ and $Q$, with $p(P)$ in the first quadrant and $p(Q)$ in the
fourth quadrant. The negations of these two sets give vectors that are mapped
into the other two quadrants.

We claim that, for a suitable choice of $P$ and $Q$, there is a permutation
$\alpha$ of the basis such that for any $g \in P \cup Q$
the difference $A( \alpha g) - g$ is a non-negative linear combination
of  $\pm$ basis vectors that map to the quadrant neighboring
the quadrant of $p(g)$ across the $S$ axis.  (To help with visualization:
for most basis elements, $\alpha$
will be chosen so that $A(g) = \alpha^{-1}(g)$.)

Suppose this claim is true.
Let $\Gamma$ be a graph with a single vertex whose edges correspond
to $P \cup Q$. Let $E$ be the homomorphism $\pi_1(\Gamma) \to \R^n$
that maps the loop of an edge labeled by a vector to that vector, and
let $\tilde \Gamma$ be the corresponding covering space. ($\Gamma$ can
be visualized as an embedded graph in the torus $\R^n / \Z^n$,
and $\tilde \Gamma$ is lift to a graph in $\R^n$).  
Define a doubletrack structure $(\sigma, \tau)$ on $\Gamma$
where a $\sigma$-legal path lifted to $\tilde \Gamma$ has monotone projection 
to the $S$-axis
and a $\tau$-legal path lifts to have monotone projection to the $T$-axis.
The linear transformation $A$ maps $\Gamma$ to a new graph $A(\Gamma)$ 
mapped to $\R^n$. Using the claim above, we can construct
a bizipping that slides the image edges back to the originals.

%\sarah commented this out
%\bigskip
%\dots {\bf TO BE FILLED IN}:
%the rest of the description and proof for the doubletrack.  A little more on zipping.
\end{proof}
%\nocite{KOR}

\bibliography{Entropy.bib}{}

\begin{thebibliography}{10}

\bibitem{ALM}
L.~Alsed{\`a}, J.~Llibre, and M.~Misiurewicz.
\newblock {\em Combinatorial dynamics and entropy in dimension one}, volume~5
  of {\em Advanced Series in Nonlinear Dynamics}.
\newblock World Scientific Publishing Co. Inc., River Edge, NJ, second edition,
  2000.

\bibitem{bertrand}
A.~Bertrand.
\newblock D\'eveloppements en base de {P}isot et r\'epartition modulo {$1$}.
\newblock {\em C. R. Acad. Sci. Paris S\'er. A-B}, 285(6):A419--A421, 1977.

\bibitem{BestvinaHandel}
M.~Bestvina and M.~Handel.
\newblock Train tracks and automorphisms of free groups.
\newblock {\em Ann. of Math. (2)}, 135(1):1--51, 1992.

\bibitem{boyd}
D.~Boyd.
\newblock Salem numbers of degree four have periodic expansions.
\newblock In {\em Th\'eorie des nombres ({Q}uebec, {PQ}, 1987)}, pages 57--64.
  de Gruyter, Berlin, 1989.

\bibitem{boyd2}
D.~Boyd.
\newblock On the beta expansion for {S}alem numbers of degree {$6$}.
\newblock {\em Math. Comp.}, 65(214):861--875, $S$29--$S$31, 1996.

\bibitem{CH}
A.~de~Carvalho and T.~Hall.
\newblock Unimodal generalized pseudo-{A}nosov maps.
\newblock {\em Geometry \& Topology}, 8:1127--1188, 2004.

\bibitem{gelfond}
A.~O. Gelfond.
\newblock A common property of number systems.
\newblock {\em Izv. Akad. Nauk SSSR. Ser. Mat.}, 23:809--814, 1959.

\bibitem{LindConversePerron}
D.~A. Lind.
\newblock The entropies of topological {M}arkov shifts and a related class of
  algebraic integers.
\newblock {\em Ergodic Theory Dynam. Systems}, 4(2):283--300, 1984.

\bibitem{MilnorThurstonKneading}
J.~Milnor and W.~Thurston.
\newblock On iterated maps of the interval.
\newblock In {\em Dynamical systems ({C}ollege {P}ark, {MD}, 1986--87)}, volume
  1342 of {\em Lecture Notes in Math.}, pages 465--563. Springer, Berlin, 1988.

\bibitem{MisiurewiczSzlenk2}
M.~Misiurewicz and W.~Szlenk.
\newblock Entropy of piecewise monotone mappings.
\newblock In {\em Dynamical systems, {V}ol. {II}---{W}arsaw}, pages 299--310.
  Ast\'erisque, No. 50. Soc. Math. France, Paris, 1977.

\bibitem{MisiurewiczSzlenk1}
M.~Misiurewicz and W.~Szlenk.
\newblock Entropy of piecewise monotone mappings.
\newblock {\em Studia Math.}, 67(1):45--63, 1980.

\bibitem{Nekrash}
V.~Nekrashevych.
\newblock Hyperbolic groupoids: definitions and duality, 2011.
\newblock arXiv:1101.5603v1.

\bibitem{salem}
R.~Salem.
\newblock A remarkable class of algebraic numbers. {P}roof of a conjecture of
  {V}ijayaraghavan.
\newblock {\em Duke Math. Journal}, 11:103--108, 1944.

\bibitem{schmidt}
K.~Schmidt.
\newblock On periodic expansions of {P}isot numbers and {S}alem numbers.
\newblock {\em Bull. London Math. Soc.}, 12(4):269--278, 1980.

\end{thebibliography}
\bibliographystyle{abbrv}

%\newpage
%
%{\sarah Hi Bill, here are some questions...
%
%\begin{itemize}
%\item p4, 3rd paragraph, discussion of the Perron-Frobenius Theorem for non-negative matrices. It is written {\em{If the matrix is ergodic, there is a unique strictly positive eigenvector; its eigenvalue is [strictly] positive.}} Why is this true? Perron-Frobenius theory generalizes to non-negative matrices, but if the matrix is not *irreducible*, then the entries of the eigenvector in question might be $0$. Are ergodic matrices automatically irreducible? If not, how do you know that the eigenvector is unique, and that the entries are always positive? 
%
%\item p6, 3rd paragraph, why is the exotic locus in Figure 1 path connected? 
%
%\item p9, 3rd paragraph: Why is there a nonempty $(d-2)$-dimensional set of $\lambda$-uniform expanders for every $1<\lambda<d$? 
%
%\item p27, I made a remark which contains the statements below. Is this okay?
%``\dots {\bf TO BE FILLED IN}: \dots
%Statement and proof that the set of friends changes continuously, regarded as an
%atomic measure on the parameter space for candidate systems, in the weak topology.  Generalization of: roots of power series with bounded coefficients are isolated and depend continuously on coefficients.''
%
%\item The proof of Theorem 1.8 is unfinished in Section 12.... what should we do with that? We could throw this part into an appendix or something. 
%
%\end{itemize}
%}

%\bibliographystyle{math} \bibliography{Entropy.bib}
\end{document}